\documentclass{amsart}

\usepackage[margin=1in]{geometry}
\usepackage{amsmath,amsthm,amssymb,amstext}
\usepackage{mathrsfs}
\usepackage{enumitem}
\usepackage[dvipsnames]{xcolor}
\usepackage{tikz-cd}
\usepackage{hyperref}
\hypersetup{colorlinks = true,
	linkcolor = Mahogany,
	urlcolor = Cerulean,
	citecolor = Bittersweet,
    pdftitle={Elliptic curves and their principal homogeneous spaces: splitting Severi--Brauer varieties},
    pdfauthor={Eoin Mackall, Nick Rekuski}
    }

\DeclareFontFamily{U}{wncy}{}
\DeclareFontShape{U}{wncy}{m}{n}{<->wncyr10}{}
\DeclareSymbolFont{mcy}{U}{wncy}{m}{n}
\DeclareMathSymbol{\Sh}{\mathord}{mcy}{"58} 

\title[elliptic curves and splitting Severi--Brauer varieties]{Elliptic curves and their principal homogeneous spaces: splitting Severi--Brauer varieties}
\author{Eoin Mackall}
\address{Department of Mathematics, University of California, San Diego, La Jolla, CA USA}
\email{etmackall@ucsd.edu}

\author{Nick Rekuski}
\address{Department of Mathematics, Wayne State University, Detroit, MI USA}
\email{rekuski@wayne.edu}
\thanks{The second author is supported by the U.S. Department of Energy, Office of Science, Basic Energy Sciences, under Award Number DE-SC-SC0022134 and an OVPR Postdoctoral Award at Wayne State University.}


\date{January 18, 2024}

\newtheorem{theorem}{Theorem}
\newtheorem{proposition}[theorem]{Proposition}
\newtheorem{lemma}[theorem]{Lemma}
\newtheorem{corollary}[theorem]{Corollary}

\theoremstyle{remark}
\newtheorem{remark}[theorem]{Remark}
\newtheorem{example}[theorem]{Example}

\theoremstyle{definition}
\newtheorem{definition}[theorem]{Definition}

\numberwithin{theorem}{section}

\newtheorem{thmx}{Theorem}

\newtheorem{exmpx}[thmx]{Example}

\begin{document}

\maketitle

\begin{abstract}
We consider the question: which elliptic curves appear as the Jacobian of a smooth curve of genus one splitting a Severi--Brauer variety? We provide three new examples.

First, we show that if $E$ is any elliptic curve over an algebraically closed field $k$ and if $F/k$ is a perfect field extension, then there exists a principal homogeneous space for $E_F$ splitting a Severi--Brauer variety $X=\mathrm{SB}(A)$ over $F$ if and only if $A$ is Brauer equivalent to a cyclic algebra. Along the way, we also give a uniform proof of a generalization of results due to O'Neil, Clark and Sharif, and Antieau and Auel.

Second, we give an example of an elliptic curve $E$ over a field $k$ together with a central simple algebra $A/k$ of degree $4$ such that $E$ is the Jacobian of a smooth genus one curve $C$ embedded in the Severi--Brauer variety $X=\mathrm{SB}(A)$ as a degree 8 curve and such that $E$ is not the Jacobian of any genus one curve of smaller degree contained in $X$. Our example is, in some sense, as small as possible in both dimension and arithmetic complexity.

Third, we show that for any odd integer $n\geq 3$ there is a central simple algebra $A$ of degree $n$ over a local field $k$ which is split by the principal homogeneous space of an elliptic curve $E/k$ but not by any principal homogeneous space for any quadratic twist of $E$. This generalizes a recent result of Saltman in the case of surfaces to arbitrary even dimension.
\end{abstract}

\tableofcontents
\addtocontents{toc}{\protect\setcounter{tocdepth}{1}} 

\section{Introduction}
Given a Severi--Brauer variety $X$, does there exist a smooth and projective curve $C$ on $X$ of genus one? If the index of $X$ is at most 5, then earlier work on this problem by de Jong and Ho \cite{MR3091612} shows that it is always possible to find such a curve $C$ so long as the dimension of $X$ allows it (i.e.\ if $\mathrm{dim}(X)>1$).
Their result is special in the sense that it relies on a general description of moduli spaces of geometrically elliptic normal curves on Severi--Brauer varieties of low dimensions. In higher dimensions, such nice descriptions of these moduli spaces no longer hold, so new techniques are needed to analyze this question in general.

Closely related to the problem of embedding a genus one curve inside a Severi--Brauer variety $X$ is the \textit{a priori weaker} problem of splitting the associated Brauer class of $X$ by the function field of such a curve. On this front, recent work of Antieau and Auel \cite{https://doi.org/10.48550/arxiv.2106.04291} has shown that over a global field $k$ one can always split Brauer classes of index $6$, $7$, or $10$ by the function field of some genus one curve defined over $k$. More classically, Roquette \cite{MR0201435} has shown that over a $p$-adic local field $k$, a Brauer class of index $n$ is split by a genus one curve $C$ over $k$ if and only if the index of $C$ is a multiple of $n$. Since genus one curves of every index exist over $p$-adic local fields, Roquette completely resolves the splitting problem in this case.

We show in Proposition \ref{prop: splitandembedd} below that these two problems -- \textit{embedding a genus one curve in $X$} and \textit{splitting $X$ by a genus one curve} -- are actually equivalent so long as $\mathrm{dim}(X)\geq 3$. Essentially, if a smooth and projective curve $C$ of genus one splits $X$, then $C$ will embed as a geometrically elliptic normal curve inside \textit{some} Severi--Brauer variety $Y$ which is Brauer equivalent to $X$. By either projecting or embedding, one can then get to every other Brauer equivalent Severi--Brauer variety which has dimension at least $3$.

In general, the embedding problem still seems quite hard. However, if one \textit{rigidifies} the problem by fixing from the start an elliptic curve $E$, then one can say much more.

\begin{thmx} Let $k$ be a field, let $E/k$ be an elliptic curve, and fix an $E$-torsor $C$. For simplicity, suppose that $\alpha\in \mathrm{Br}(k)$ is a Brauer class having index at least $4$. Then the following statements are equivalent:
\begin{enumerate}
\item $C$ splits the Brauer class $\alpha$;
\item $C$ embeds in $X=\mathrm{SB}(A)$ for every central simple $k$-algebra $A$ with $\alpha=[A]$;
\item there exists an algebra $B$ with $[B]=\alpha$ such that $C$ embeds in $Y=\mathrm{SB}(B)$ as a geometrically elliptic normal curve.
\end{enumerate}
\end{thmx}

For a fixed elliptic curve $E$, results of O'Neil, Clark, and many others \cite{MR1924106, MR2163913, MR2384334} allow us to identify the set of all embeddings $C\subset X$ of an $E$-torsor $C$ as a geometrically elliptic normal curve inside a Severi--Brauer variety $X$ of dimension $n$, varying over all $C$ and all $X$, with the Galois cohomology group $\mathrm{H}^1(k,E[n+1])$ of the $(n+1)$-torsion subscheme $E[n+1]$ of $E$. In Theorem A above, we can then add the equivalent statement:

\begin{enumerate}[label= (4)]
\item there exists an element $\xi\in \mathrm{H}^1(k,E[n+1])$ such that $\Delta(\xi)=\alpha$ where $\Delta:\mathrm{H}^1(k,E[n+1])\rightarrow \mathrm{Br}(k)$ is the \textit{Clark--O'Neil obstruction map} sending an element $[C\subset X]$ to $[X]\in \mathrm{Br}(k)$.
\end{enumerate}

\noindent Effectively, the above says that if one is willing to fix the elliptic curve $E$ from the beginning, then determining if an $E$-torsor $C$ embeds into a Severi--Brauer $X$ is equivalent to determining if the Brauer class of $X$ is in the image of some maps in Galois cohomology. This is the perspective that we take in the present paper and, in many ways, we view the work of Antieau--Auel and of Saltman \cite{https://doi.org/10.48550/arxiv.2105.09986} as having a very similar philosophy.

Our first main result is the following:

\begin{thmx}[Theorem \ref{thm:fulllevelinf}] Let $E$ be any elliptic curve over an algebraically closed field $k$. Let $F/k$ be a field extension and assume $F$ is perfect. 

Then there exists a principal homogeneous space for $E_F$ admitting a closed embedding into a Severi--Brauer variety $X=\mathrm{SB}(A)$ over $F$ if and only if $A$ is Brauer equivalent to a cyclic algebra.
\end{thmx}

The universal division algebra of square index at least $4$ over a field of characteristic zero is not Brauer equivalent to a cyclic algebra as Example \ref{exmp: wadsworth} shows, so such curves should embed in relatively few varieties. That such principal homogeneous spaces actually exist in the statement of Theorem B can be deduced from the theorem of Clark and Sharif \cite[Theorem 10]{MR2592017} on the cyclicity of the obstruction map for elliptic curves which have all of their torsion points defined rationally. It can also be deduced from the results of \cite{https://doi.org/10.48550/arxiv.2106.04291} which generalize this statement in a few ways (e.g.\ to level structures, and to the splitting of gerbes). 

We give in Section \ref{sec:symbols} another proof that the obstruction map is cyclic under some assumptions on the torsion subgroup scheme of the fixed elliptic curve. Our proof is distinct from both the proofs of Clark--Sharif and Antieau--Auel and goes along the lines of an original (but incorrect) approach of O'Neil \cite{MR1924106}. Benefits of our approach include an explicit cocycle description for the error term found by Antieau--Auel in the even degree case (see Theorem \ref{thm:antieauauel}) and a new proof that this error term vanishes under an assumption of higher level structure as in the work of Clark--Sharif (see Corollary \ref{thm:2nstructure}). Using these results we also give:\\

\begin{exmpx} We give an example of an elliptic curve $E$ over $\mathbb{Q}_2((t))$ together with a central simple algebra $A/k$ of degree $4$ such that $E$ is the Jacobian of a smooth genus one curve $C$ embedded in the Severi--Brauer variety $X=\mathrm{SB}(A)$ as a degree 8 curve and such that $E$ is not the Jacobian of any genus one curve of smaller degree contained in $X$.\end{exmpx}

The elliptic curve that we consider in this example is the universal elliptic curve with full level-$8$ torsion. The algebra $A$ that we consider was constructed earlier by Brussel \cite{MR4072790} and is a noncyclic degree 4 algebra with a degree 8 cyclic splitting field.

We should point out that, while it's easy to construct smooth curves of genus one which embed in Severi--Brauer varieties as curves of a certain 
non-minimal degree and not of any smaller degree (see Remark \ref{rmk:nonellnorm}), it is difficult to produce examples where there are no principal homogeneous spaces at all for the Jacobian of these curves which admit an embedding of smaller degree. For example, we show in Remark \ref{rmk: embedpadic} that no such examples can exist over $p$-adic local fields and, so, our example is minimal in some arithmetic sense.

Lastly, we show:

\begin{thmx}[Corollary \ref{cor: exmpc}] For any odd integer $n\geq 3$, there exists a central simple algebra $A$ of degree $n$ over a local field $k$ which is split by the principal homogeneous space of an elliptic curve $E/k$ but not by any principal homogeneous space for any quadratic twist of $E$. 

In fact, we can produce infinitely many such fields $k$ and infinitely many such curves $E$ for a given $k$ which fit into an example $(k,A,E)$ for some $A$ as above.
\end{thmx}

Saltman \cite{https://doi.org/10.48550/arxiv.2105.09986} gives similar examples in the case of a Severi--Brauer surface. Theorem D can be seen as a generalization of these examples to higher dimension. The key observation allowing us to make this generalization is that Saltman's techniques are secretly making heavy use of the correspondence between what we call ``Severi--Brauer diagrams" and the Galois cohomology of certain torsion subschemes of an elliptic curve. Once one takes this fact into account, Saltman's examples are gotten from observing certain elliptic curves admit principal homogeneous spaces of a given period while their quadratic twists do not. We're then able to prove our result by analyzing the inflation-restriction sequence of cohomology, see Theorem \ref{thm:quadtwist}.

\subsection*{Outline}
This document is structured as follows. Section \ref{sec: prelims} is preliminary and is devoted to a proof of Theorem A. In this section we set-up the correspondence between Severi--Brauer diagrams under an elliptic curve and Galois cohomology of torsion subschemes of the curve, between the splitting problem and the embedding problem, and between splitting by genus one curves and the existence Severi--Brauer diagrams. We conclude by showing how one can construct all of the obstruction maps which have been considered in the literature thus far and we show how to relate them to one another. While this section is mostly a survey in nature, we have opted to include all proofs for this section as they are sometimes missing detail, or are incorrect, in the existing literature.

In Section \ref{sec:symbols}, we give our description of the obstruction map associated to an $n$-torsion cyclic subgroup scheme of an elliptic curve in the case that the characteristic does not divide $n$ and under the assumption that the elliptic curve has full level-$n$ structure. Noteworthy here is both our description of the error term in the case that $n$ is even and our description of the obstruction map without the assumption of a symplectic full level-$n$ structure. We end this section with the constructions of both Theorem B and Example C. An expert that's interested in these examples may be able to start reading from subsection \ref{ss: cyclicsplit} without having to invest too much time into the previous sections.

In Section \ref{sec: conj}, we show how one can reduce from the case of an arbitrary elliptic curve to one with full level-$n$ structure (where everything is understood) by restriction. We interpret the image of the restriction map (to the field where the elliptic curve acquires all of its $n$-torsion points) as a set of isomorphism classes of Galois algebras which are invariant under the conjugation action. This allows us to obstruct the existence of principal homogeneous spaces of a given period by showing the nonexistence of certain Galois extensions of the base field. An easy argument then gives Theorem D.

\subsection*{Notation and Conventions}
  \begin{itemize}
    \item{
     We fix $k$ be a base field and $\overline{k}$ a fixed algebraic closure. Inside $\overline{k}$ we also fix a separable closure $k^s$ of $k$. We typically write $F/k$, $\overline{F}$, and $F^s$ respectively for a field extension of $k$
    }
    \item{
      We let $E$ be an elliptic curve over $k$ with group operation $m:E\times E\to E$
    }
    \item{
      The absolute Galois group of $k$ (resp.\ of a field $F/k$) is written as $G_{k}=\mathrm{Gal}(k^s/k)$ (resp.\ as $G_{F}$).
    }
    \item{
      We typically write $A$ for a $k$-central simple $k$-algebra and $X=\operatorname{SB}(A)$ for the Severi--Brauer variety associated to $A$.
    }
  \end{itemize}

\subsection*{Acknowledgements}
The authors wish to thank Asher Auel, Benjamin Antieau, Danny Krashen, and David Saltman for conversations about various parts of this work. The authors would also like to thank Adrian Wadsworth for the very detailed explanation of the construction in Example \ref{exmp: wadsworth}.

\addtocontents{toc}{\protect\setcounter{tocdepth}{2}} 
  
\section{Severi-Brauer diagrams and obstruction maps}\label{sec: prelims}
Throughout this section we choose a fixed, but arbitrary, base field $k$. For any elliptic curve $E$ over $k$ and for any integer $n\geq 2$ we write $E[n]$ for the kernel subgroup scheme of the multiplication-by-$n$ map of $E$. In this section we describe the relationship between elements of the Galois cohomology group $\mathrm{H}^1(G_k, E[n](k^s))$, where $k^s$ is a separable closure of $k$, and morphisms from $E$-torsors to Severi--Brauer varieties.

\subsection{Severi--Brauer diagrams}
Here we recall an interpretation of the elements of the cohomology group $\mathrm{H}^1(G_k,E[n](k^s))$ in terms of Severi-Brauer diagrams. This interpretation has been used frequently in some contexts, e.g.\ on the period-index problem for genus one curves \cite{MR1924106,MR2163913,MR2592017} and in explicit $n$-descent for elliptic curves \cite{MR2384334}, but it's only been used infrequently in the study of the Brauer group.

\begin{definition}
    \label{def:severiBrauerDiagram}
    Fix an elliptic curve $E/k$ and an integer $n\geq 2$.
    We define a \textit{degree-$n$ Severi-Brauer diagram under $E$} as a triple
    \[
        (C,\,\,\, \mu:E\times C\rightarrow C,\,\,\, f:C\rightarrow X)
    \]
    consisting of a curve $C$, a fixed principal homogeneous structure $\mu$ for $E$ on $C$, and a morphism $f:C\rightarrow X$ to a Severi--Brauer variety $X$ of dimension $\mathrm{dim}(X)=n-1$ satisfying:
    \begin{itemize}
    \item if $n=2$, then $f$ is a degree-$2$ cover of $X$;
    \item if $n>2$, then $f$ is a closed immersion and $\mathrm{deg}(f(C))=n$.
    \end{itemize}
    The degree of a closed subvariety of $X$ is calculated as the degree of the corresponding projective subvariety under any isomorphism $X_{F}\cong\mathbb{P}^{n-1}_{F}$ over some (equivalently any) splitting field $F/k$ of $X$.

    Two degree-$n$ Severi--Brauer diagrams
    \[
      (C,\,\,\,\mu:E\times C\rightarrow C,\,\,\,f:C\rightarrow X)\quad\mbox{and}\quad (C',\,\,\,\mu':E\times C'\rightarrow C',\,\,\,f':C'\rightarrow X')
    \]
    are said to be \textit{equivalent} if there exist $k$-isomorphisms $\phi:C\cong C'$ and $\psi:X\cong X'$ such that the following two diagrams of $k$-schemes commute:
    \[
      \begin{tikzcd}
        E\times C\arrow["\mu"]{r}\arrow["\operatorname{id}\times\phi"]{d} & C\arrow["\phi"]{d} \\
        E\times C'\arrow["\mu'"]{r} & C'
      \end{tikzcd}
      \qquad
      \begin{tikzcd}
        C\arrow["f"]{r}\arrow["\phi"]{d} & X\arrow["\psi"]{d} \\ C'\arrow["f'"]{r} & X'
      \end{tikzcd}.
    \]
Note this implies any fixed degree-$n$ Severi--Brauer diagram $(C,\, \mu,\, f:C\rightarrow X)$ is equivalent to the diagram $(C,\, \mu,\, \phi\circ f:C\rightarrow X)$ for any $k$-automorphism $\phi\in \mathrm{Aut}(X/k)$.

We write $\pi_0(E-\mathsf{SB}_n)(F)$ to denote the set of equivalence classes of degree-$n$ Severi--Brauer diagrams under $E$ that are defined over a field extension $F/k$ (appropriately defined).
  \end{definition}
  
Given any fixed degree-$n$ Severi--Brauer diagram $\xi=(C,\,\mu:E\times C\rightarrow C,\, f:C\rightarrow X)$ over $k$, any other degree-$n$ Severi--Brauer diagram over $k$ can be realized as a twist of $\xi$ as we now explain. Classically, a continuous $1$-cocycle $\alpha:G_k\rightarrow E(k^s)$ determines both the data of a curve $C'$ defined over $k$ and a principal homogeneous structure $\mu':E\times C'\rightarrow C'$ from $(C,\mu)$ by Galois descent. The curve $C'$ has the property that there is a $k^s$-scheme isomorphism $\phi:C'_{k^s}\cong C_{k^s}$ and, for any $\sigma\in G_k$, there is a commuting square 
  \[
    \begin{tikzcd}
      C'_{k^s}\arrow["\phi"]{r}\arrow["\sigma"]{d} & C_{k^s}\arrow["\alpha_\sigma \circ \sigma"]{d} \\
      C'_{k^s}\arrow["\phi"]{r} & C_{k^s}.
    \end{tikzcd}
  \]
  The principal homogeneous structure $\mu'$ on $C'$ is descended from the composition \[E_{k^s}\times C'_{k^s}\xrightarrow{\mathrm{id}\times \phi} E_{k^s}\times C_{k^s}\xrightarrow{\mu_{k^s}} C_{k^s}\xrightarrow{\phi^{-1}} C'_{k^s}.\]
  Moreover, the isomorphism $\phi:C'_{k^s}\cong C_{k^s}$ is naturally an isomorphism of $E_{k^s}$ principal homogeneous spaces.

  Moreover, if $n\geq 2$ is a fixed integer and if the $1$-cocycle $\alpha$ takes values in the $k^s$-points $E[n](k^s)$ of the $n$-torsion subscheme $E[n]$ of $E$, then $\alpha$ additionally determines the equivalence class of a Severi--Brauer variety $X'$ of dimension $n-1$ and a (nonunique) morphism $f':C'\rightarrow X'$ which determines a degree-$n$ Severi--Brauer diagram under $E$ up to equivalence. To be precise, let $t\in \mathbf{Pic}_C^n(k)$ be the image of $\mathcal{O}(1)\in \mathbf{Pic}^1_X(k)$ under the pullback induced by $f$. Pulling back along $\phi$ yields a map \[\phi^*:\mathbf{Pic}_{C_{k^s}}(k^s)\rightarrow\mathbf{Pic}_{C'_{k^s}}(k^s)\] that determines a $k^s$-point $\phi^*(t)\in\mathbf{Pic}^n_{C'_{k^s}}(k^s)$. Then, assuming $\alpha(G_k)\subset E[n](k^s)$, one can observe (as we do below) that the $k^s$-point $\phi^*(t)$ is $G_k$-invariant and, so, it descends to a $k$-rational point of $\mathbf{Pic}^n_{C'}(k)$. Associated to any such $k$-point is a morphism from $C$ to a Severi--Brauer variety $X'$ of dimension $n-1$, unique up to the action of $\mathrm{Aut}(X')$, and this defines an equivalence class of a degree-$n$ Severi--Brauer diagram under $E$ and over $k$. (For more on the correspondence between $k$-rational points on the Picard scheme $\mathbf{Pic}_{C'}$ of $C'$ and morphisms to Severi--Brauer varieties, see \cite[Section (5.4)]{MR0244271} for an original source; see \cite{MR3644253} for a more recent exposition.)
  
  It remains to explain why $\phi^*(t)$ is $G_k$-invariant. To do this, we note that since $E$ acts on $C$ and $C'$ via $\mu$ and $\mu'$, there are induced actions of the opposite group $k$-scheme $E^{op}$ on each of the connected components of the Picard schemes $\mathbf{Pic}_C$ and $\mathbf{Pic}_{C'}$. Over the separable closure $k^s$ of the field $k$, the action on $\mathbf{Pic}_C\times_k k^s$ is defined on $T$-points, for any $k^s$-scheme $T$, by \[\hat{\mu}_{k^s}:E^{op}_{k^s}(T)\times\mathbf{Pic}_{C_{k^s}}(T)\rightarrow \mathbf{Pic}_{C_{k^s}}(T) \quad \mbox{defined by}\quad (x,\mathcal{L})\mapsto \hat{\mu}_{k^s}(x,\mathcal{L}):=T_x^*\mathcal{L}\] where $T_x=(p_1,\mu\circ (x\times \mathrm{id}_C)):T\times_{k^s} C_{k^s}\rightarrow T\times_{k^s} C_{k^s}$ is the left translation-by-$x$ map on $C_T:=T\times_{k^s} C_{k^s}$ and we write $p_1$ for the first projection. This action morphism is Galois equivariant for the appropriate \textit{right} Galois actions, meaning that for each $g\in G_k$ there is a commuting diagram \[\begin{tikzcd}[column sep = huge] E_{k^s}^{op}\times \mathbf{Pic}_{C_{k^s}}\arrow["g_{E^{op}}\times g_{\mathbf{Pic}_C}"]{r}\arrow["\hat{\mu}_{k^s}"]{d} & E_{k^s}^{op}\times \mathbf{Pic}_{C_{k^s}}\arrow["\hat{\mu}_{k^s}"]{d} & (x,\mathcal{L})\arrow[mapsto]{r}\arrow[mapsto]{d} & (g_E^{-1}x, g_C^*\mathcal{L})\arrow[mapsto]{d} \\
  \mathbf{Pic}_{C_{k^s}} & \mathbf{Pic}_{C_{k^s}}\arrow["g_{\mathbf{Pic}_C}^{-1}"]{l} & T_x^*\mathcal{L}  & T_{(g_E^{-1}x)}^* g_C^*\mathcal{L}\arrow[mapsto]{l}\end{tikzcd}
  \] where one uses the relations $g_{\mathbf{Pic}_C}=g_C^*$ and $T_{(g_{E}x)}=g_C\circ T_x \circ g_C^{-1}$ to obtain the square on the right above. It follows that there exists a morphism $\hat{\mu}$ defining an action of $E^{op}$ on $\mathbf{Pic}_C$ over $k$ which becomes $\hat{\mu}_{k^s}$ after extending scalars. One similarly defines the action $\hat{\mu}'$ of $E^{op}$ on $\mathbf{Pic}_{C'}$.
  
Since $E$ is commutative we can identify $E$ with $E^{op}$ and consider $\hat{\mu}$ and $\hat{\mu}'$ as actions by $E$. It follows immediately that $\phi^*$ is equivariant with respect to these actions scalar extended to the separable closure $k^s$. Further, one can check that the $n$-torsion subgroup scheme $E[n]$ acts trivially on the connected components $\mathbf{Pic}_C^n$ and $\mathbf{Pic}^n_{C'}$ (this can checked directly if $\mathrm{char}(k)\nmid n$ or by using \cite[\S23, Theorem 3]{MR0282985} in general). Hence if $\alpha(G_k)\subset E[n](k^s)$, then $\phi^*(t)$ is $G_k$-invariant as claimed. Altogether, one finds:
  
\begin{proposition}[cf.\ O'Neil, {\cite[Proposition 2.2]{MR1924106}}]
    \label{prop:severiBrauerDiagramsBijections}
    Let $(k,E,n,\xi)$ be defined as above, consider the set $\pi_0(E-\mathsf{SB}_n)(k)$ as a set pointed by $\xi$, and assume that one of the following holds:
    \begin{enumerate}[label=(F\arabic*)]
    \item\label{item: f1} $k$ is arbitrary and $\mathrm{char}(k)\nmid n$,
    \item\label{item: f2} or $k$ is perfect.
    \end{enumerate}
    Then, the above association of an equivalence class of a degree-$n$ Severi--Brauer diagram from a continuous $1$-cocycle $\alpha$ induces a well-defined bijection of pointed sets
    \begin{equation}
      \label{eq:tsbcorr}
      T_\xi:\mathrm{H}^1(G_k,E[n](k^s))\rightarrow \pi_0(E-\mathsf{SB}_n)(k).
    \end{equation} By viewing $\pi_0(E-\mathsf{SB}_n)(k)$ as the collection of equivalence classes of twists of the degree-$n$ Severi--Brauer diagram $\xi$, an explicit inverse \begin{equation}
      \label{eq:ssbcorr}S_\xi: \pi_0(E-\mathsf{SB}_n)(k)\rightarrow \mathrm{H}^1(G_k,E[n](k^s))\end{equation} to $T_\xi$ can be constructed through splitting of degree-$n$ Severi--Brauer diagrams. Moreover, both of the maps (\ref{eq:tsbcorr}) and (\ref{eq:ssbcorr}) are natural with respect to restriction to any intermediate field $k\subset F \subset k^s$.
\end{proposition}

\begin{proof}
Let $\xi'=(C',\mu',f':C'\rightarrow X')$ be any degree-$n$ Severi--Brauer diagram under $E$ which splits over a finite Galois extension $F/k$. By \textit{split}, we mean that $\xi'_F$ is equivalent to $\xi_F$ so there exist $F$-isomorphisms $\alpha,\beta$ making the diagram \[\begin{tikzcd} C'_F\arrow["f'"]{r}\arrow["\alpha"]{d} & X'_F\arrow["\beta"]{d} \\ C_F\arrow["f"]{r} & X_F  \end{tikzcd}\] commutative. For each $\sigma\in \mathrm{Gal}(F/k)$ we set $\gamma_\sigma= \alpha \circ \sigma_{C'_F} \circ \alpha^{-1} \circ \sigma_{C_F}^{-1}=\alpha\circ {^\sigma\alpha^{-1}}$ to be the given automorphism of $C$ considered as a principal homogeneous space under $E$. The map $\gamma:\mathrm{Gal}(F/k)\rightarrow \mathrm{Aut}_{phs}(C_F,\mu)=E(F)$ defined by $\gamma(\sigma)=\gamma_\sigma$ is then a $1$-cocycle. Since the embeddings $f'$ and $f$ are Galois equivariant, one can check that for each $\sigma\in \mathrm{Gal}(F/k)$, the automorphism $\gamma_\sigma$ fixes $f^*\mathcal{O}(1)=\mathcal{L}$ on $C_F$. This implies that $\gamma_\sigma\in E[n](F)$ due to \cite[\S23, Theorem 3]{MR0282985}.

For any field extension $L/k$, denote by $W_\xi(L)\subset \pi_0(E-\mathsf{SB}_n)(k)$ the subset consisting of those equivalence classes of Severi--Brauer diagrams of degree-$n$ under $E$ which split (to $\xi$) over $L$. The correspondence, defined above, sending $\xi'$ to $\gamma$  sets up an injection from $W_\xi(F)$ to the group cohomology $\mathrm{H}^1(\mathrm{Gal}(F/k),E[n](F))$. Taking limits over the filtered system of inflation maps, one gets an injection \begin{equation} \label{eq:rho} \rho:W_\xi(k^s)\rightarrow \mathrm{H}^1(G_k, E[n](k^s))\end{equation} to the continuous Galois cohomology group. We claim that under either assumption \ref{item: f1} or \ref{item: f2}: we have equality $W_\xi(k^s)=\pi_0(E-\mathsf{SB}_n)(k)$ (i.e.\ every Severi--Brauer diagram splits over the separable closure) and that the map $\rho$ is a surjection.

We first show that every Severi--Brauer diagram splits over the separable closure $k^s$ of $k$ assuming either \ref{item: f1} or \ref{item: f2}. Let $\xi'=(C',\mu',f')$ denote an arbitrary diagram of degree-$n$ under $E$. Using a point $t_C\in C(k^s)$, we can identify $E_{k^s}$ with $\mathbf{Pic}^1_{C_{k^s}}$ by sending the identity $p\in E(k^s)$ to $\mathcal{O}(t_C)$. Similarly we may identify $E_{k^s}$ with $\mathbf{Pic}^n_{C_{k^s}}$ sending $p$ to $\mathcal{O}(t_C)^{\otimes n}$. The multiplication-by-$n$ map on $E$ can then be identified with the exact sequence of group schemes \begin{equation}\label{seq: picn} 0\rightarrow E[n]_{k^s}\rightarrow \mathbf{Pic}^1_{C_{k^s}}\rightarrow \mathbf{Pic}^n_{C_{k^s}}\rightarrow 0\end{equation} where the nontrivial map on the right sends a line bundle $\mathcal{L}$ to $\mathcal{L}^{\otimes n}$. The sequence (\ref{seq: picn}) is exact on $k^s$ points by \cite[(22.15) Proposition]{MR1632779}, if one assumes \ref{item: f1}, or because $k^s$ is an algebraic closure assuming \ref{item: f2}. Thus, the map $f:C\rightarrow X$ has the property that $f_{k^s}^*\mathcal{O}(1)=\mathcal{O}(nx_C)$ for some point $x_C\in C(k^s)$. Similarly, one has that $(f'_{k^s})^*\mathcal{O}(1)=\mathcal{O}(nx_{C'})$ for a point $x_{C'}\in C'(k^s)$.

One can now construct an isomorphism $C'_{k^s}\rightarrow C_{k^s}$ of principal homogeneous spaces under $E_{k^s}$ determined entirely by sending $x_{C'}$ to $x_C$. This isomorphism extends to a splitting of $\xi'$ as desired.

Finally, we take $S_\xi$ to be the map $\rho$ of (\ref{eq:rho}). It's immediate that an explicit inverse to $\rho$ is gotten by the procedure outlined in the paragraphs below Definition \ref{def:severiBrauerDiagram}, showing that $T_\xi$ is a well-defined bijection.
\end{proof}

Among all equivalence classes of degree-$n$ Severi--Brauer diagrams under $E$, there is a seemingly \textit{most canonical} such class. Specifically, we mean the class containing any degree-$n$ Severi--Brauer diagram of the form $(E,\,\,m,\,\,\iota:E\rightarrow \mathbb{P}^{n-1})$ where $m:E\times E\rightarrow E$ is the multiplication map for $E$ and $\iota$ is a fixed morphism corresponding to a complete linear system of $\mathcal{O}(p)^{\otimes n}$ where $p\in E(k)$ is the identity. Accordingly, we call any fixed choice $\xi_0=(E,m,\iota)$ representing this class the \textit{trivial degree-$n$ Severi--Brauer diagram under $E$}. From now on, unless we state otherwise, we will always assume that the set $\pi_0(E-\mathsf{SB}_n)(k)$ is pointed by the fixed trivial diagram $\xi_0$. We will also drop subscripts and write $T=T_{\xi_0}$ and $S=S_{\xi_0}$ for the corresponding maps of Proposition \ref{prop:severiBrauerDiagramsBijections} in this case.

\begin{remark}\label{rmk: wc}
It is well known $\mathrm{H}^1(G_k,E(k^s))$ can be identified with the Weil--Ch\^atelet group $\mathrm{WC}(E/k)$ whose elements are equivalence classes of principal homogeneous spaces under $E$ and defined over $k$, see \cite[Section 2]{MR0106226} or \cite[Chapter X, Theorem 3.6]{MR0817210}.
The composition \[\mathrm{H}^1(G_k,E[n](k^s))\rightarrow \mathrm{H}^1(G_k,E(k^s))\] induced by the inclusion $E[n]\subset E$ then sits in a commutative square of pointed maps \[\begin{tikzcd}\mathrm{H}^1(G_k, E[n](k^s))\arrow{r}\arrow[equals]{d} & \mathrm{H}^1(G_k, E(k^s))\arrow[equals]{d}\\ \pi_0(E-\mathsf{SB}_n)(k)\arrow{r} & \mathrm{WC}(E/k)\end{tikzcd}\]
where the bottom map sends a degree-$n$ Severi--Brauer diagram $[(C,\mu,f)]$ to the class $[(C,\mu)]$.
\end{remark}

\begin{remark}[cf.\ {\cite[Remark 1.10]{MR2384334}}]\label{rmk:Kummer}
For any given principal homogeneous space $(C,\mu)$ under $E$, there can exist multiple, inequivalent degree-$n$ Severi--Brauer diagrams under $E$ for this principal homogeneous space. To see this, consider the K\"ummer sequence of multiplication-by-$n$ on $E$, \begin{equation} \label{eq:Kummer} 0\rightarrow E[n]\rightarrow E\xrightarrow{\cdot n} E\rightarrow 0.\end{equation}
We can identify the middle curve $E$ with the elliptic curve $\mathbf{Pic}^1_E$, pointed by $\mathcal{O}(p)$, and the rightmost curve $E$ with the elliptic curve $\mathbf{Pic}^n_E$, pointed by $\mathcal{O}(p)^{\otimes n}$, via maps \[\phi_1:E\rightarrow \mathbf{Pic}^1_E,\,\,\, x\mapsto \mathcal{O}(x) \quad \mbox{and}\quad \phi_n:E\rightarrow \mathbf{Pic}^n_E,\,\,\, x\mapsto \mathcal{O}(x)\otimes\mathcal{O}(p)^{\otimes (n-1)}.\] Assuming that (\ref{eq:Kummer}) is exact on $k^s$-points, there is then an associated ladder of exact sequences in cohomology \begin{equation} \label{eq: cohKummer}\begin{tikzcd} 
E(k)\arrow[equals, "\phi_n(k)"]{d}\arrow{r} & \mathrm{H}^1(G_k,E[n](k^s))\arrow[equals]{d}\arrow{r} & \mathrm{H}^1(G_k,E(k^s))[n]\arrow[equals, "\phi_{1,*}"]{d} \\
\mathbf{Pic}^n_E(k)\arrow["\delta"]{r} & \mathrm{H}^1(G_k, E[n](k^s))\arrow{r} & \mathrm{H}^1(G_k, \mathbf{Pic}_E^1(k^s))[n]
\end{tikzcd}\end{equation} that can be interpreted using Remark \ref{rmk: wc}. Namely, the equivalence class of a degree-$n$ Severi--Brauer diagram $[(C,\mu,f)]$ maps to the identity of the group $\mathrm{H}^1(G_k,E(k^s))$ if and only if $(C,\mu)$ is isomorphic with $(E,m)$ as a principal homogeneous space under $E$ if and only if $C(k)\neq \emptyset$ if and only if $C\cong E$.

The map $\delta$ of (\ref{eq: cohKummer}) is defined by sending a point $t\in \mathbf{Pic}_E^n(k)$, corresponding to a degree-$n$ line bundle $\mathcal{L}_t$ on $E$, to the element corresponding to the class of the Severi--Brauer diagram $[(E,m,f_t)]$ where $f_t:E\rightarrow \mathbb{P}^{n-1}$ is any morphism determined by the complete linear system of $\mathcal{L}_t$. The element $\delta(t)$ is trivial in $\mathrm{H}^1(G_k,E[n](k^s))$ if and only if $\mathcal{L}_t\cong \mathcal{O}(q)^{\otimes n}$ for a point $q\in E(k)$.
\end{remark}

\subsection{Morphisms to Severi--Brauer varieties}
Let $k$ be a field and $C$ a curve over $k$. If $X$ is a Severi--Brauer variety over $k$ admitting a morphism $C\rightarrow X$, then the Severi--Brauer variety $X_{k(C)}$, over the function field $k(C)$ of $C$, is necessarily split since it contains a $k(C)$-rational point. Conversely, if $C$ is smooth, proper, and if $X$ splits over the function field $k(C)$, then any $k(C)$-rational point on $X_{k(C)}$, corresponding to a rational map $C\dashrightarrow X$, extends to a morphism $C\rightarrow X$.

When $C$ is a smooth, proper, and geometrically connected curve having arithmetic genus $g(C)=1$, then both of the conditions above are equivalent to the existence of a \textit{containment} $C\subset X$ so long as $\mathrm{dim}(X)\geq 3$. We make this more precise with the following:

\begin{proposition}\label{prop: splitandembedd}
Let $k$ be any field, let $X$ be a Severi--Brauer variety over $k$, and let $C$ be a smooth, proper, geometrically connected $k$-curve of genus $g(C)=1$. Write $E=\mathbf{Pic}_C^0$ for the connected component of the Picard scheme of $C$ containing the identity, and let $\mu:E\times C\rightarrow E$ be a principal homogeneous structure of $E$ on $C$.
If $\mathrm{dim}(X)\geq 3$, then the following statements are equivalent:
\begin{enumerate}[label=(A\arabic*)]
\item\label{prop: ds} $C$ splits $X$, i.e.\ $[X]\in \mathrm{Br}(k)$ is in the kernel of the restriction from $k$ to $k(C)$.
\item\label{prop: dci} There exists a closed immersion $i:C\rightarrow X$.
\item\label{prop: dsbd} There exists an integer $d\geq 2$, a Severi--Brauer variety $Y$ with $[Y]=[X]\in \mathrm{Br}(k)$, and a morphism $f:C\rightarrow Y$ such that $(C,\mu,f)$ defines a degree-$d$ Severi--Brauer diagram under $E$.
\end{enumerate}
\end{proposition}

\begin{proof}
The proposition is well-known, or readily deducible, in the case that $X$ is split, equivalently $X\cong \mathbb{P}^{n-1}$. We can therefore assume that $X$ is not split throughout the proof. If $C$ splits $X$, then $X_{k(C)}(k(C))\neq \emptyset$. Correspondingly, there is some morphism $j:C\rightarrow X$. Pulling back, one gets an associated $k$-rational point \begin{equation}\label{eq:point} j^*:\mathbf{Pic}^1_X(k)\rightarrow \mathbf{Pic}^d_C(k)\end{equation} as the image of the line bundle $\mathcal{O}(1)$ on $X_{k^s}$ considered as the sole element of $\mathbf{Pic}^1_X(k)$, and for some $d\in \mathbb{Z}$. Since $X$ is not split, we find $d\geq 2$.

Since $d\geq 2$, the corresponding $k$-rational point in $\mathbf{Pic}_C^d(k)$ defines a morphism $f:C\rightarrow Y$ to a Severi--Brauer variety $Y$ that's Brauer equivalent to $X$ and with $\mathrm{dim}(Y)=d-1$; geometrically, this morphism corresponds to the complete linear system defined by $j^*\mathcal{O}(1)$ on $C_{k^s}$. It follows from the Riemann-Roch theorem and \cite[Chapter 4, Corollary 3.2]{MR0463157} that $(C,\mu,f)$ defines a degree-$d$ Severi--Brauer diagram under the elliptic curve $E$, hence \ref{prop: ds} implies \ref{prop: dsbd}.

Let $P$ be a Severi--Brauer variety that's Brauer equivalent to $X$ with $\mathrm{dim}(P)$ minimal. As such, there is a closed immersion $P\rightarrow Y$ which realizes $P$ as a twisted linear subvariety of $Y$. We can view $P$ as the Severi--Brauer variety associated to a division $k$-algebra $D$, i.e.\ $P$ represents the functor whose $R$-points, for a finitely generated $k$-algebra $R$, are given as (cf.\ \cite[Definition 10.16]{MR1758562}) \[\mathrm{SB}(D)(R)=\{N\subset D_R :  N \mbox{ is an } R\mbox{-flat, right } D_R\mbox{-submodule of } D_R \mbox{ with } \mathrm{rk}_R(N)=\mathrm{deg}(D)\}.\] We can view $Y$ similarly as the Severi--Brauer variety associated to a matrix ring $M_n(D)$, for some $n\geq 1$, so that $Y$ represents the functor with $R$-points \[\mathrm{SB}(M_n(D))(R)=\{N\subset D_R^{\oplus n} : N \mbox{ is an } R\mbox{-flat, right } D_R\mbox{-submodule of } D_R^{\oplus n} \mbox{ with } \mathrm{rk}_R(N)=\mathrm{deg}(D)\}.\] With the functor notation, a closed immersion $P\rightarrow Y$ corresponds to an inclusion $i:D\subset D^{\oplus n}$ of right $D$-modules. Associated to the inclusion $i$ is also a quotient, say $q:D^{\oplus n}\rightarrow D^{\oplus n}/i(D)$.

For any finitely generated $k$-algebra $R$, and for any $R$-flat right $D_R$-submodule $N\subset D_R^{\oplus n}$ with rank $\mathrm{rk}_R(N)=\mathrm{deg}(D)$ that's not contained in $i(D)_R$, the image $q_R(N)$ inside $D_R^{\oplus n}/i(D)_R$ is again an $R$-flat right $D_R$-submodule with rank $\mathrm{rk}_R(q_R(N))=\mathrm{deg}(D)$. This defines a projection map $Y\setminus P\rightarrow Y'$ to another Severi--Brauer variety $Y'$, nonempty whenever $Y\neq P$, which is also Brauer equivalent to $X$ and which has smaller dimension than $Y$ (in fact, $\mathrm{dim}(Y')=\mathrm{dim}(Y)-\mathrm{dim}(P)-1$ exactly).

Now we work in cases. Assume first that $\mathrm{dim}(P)>2$. If $Y=P$, so that $d=\mathrm{dim}(P)+1$, then $P$ includes into any Brauer-equivalent Severi--Brauer variety. Hence there exists a closed immersion $C\subset Y=P\subset X$. On the other hand if $Y\neq P$, then we proceed as follows. Geometrically, the map $Y\setminus P\rightarrow Y'$ is a linear projection and, therefore, for a generic choice of $P\subset Y$ with $C\cap P=\emptyset$, projection induces a closed immersion $C\subset Y\setminus P\rightarrow Y'$. Repeating this process we eventually find closed immersions $C\subset P\subset X$ as above.

Assume now that $\mathrm{dim}(P)\leq 2$. Regardless of $d$, we first consider the $k$-rational point on $\mathbf{Pic}_C^{5d}(k)$ gotten from the rational point in (\ref{eq:point}) above by applying the multiplication-by-$5$ map $$\mathbf{Pic}^d_C(k)\xrightarrow{\Delta} \prod_{i=1}^5\mathbf{Pic}^d_C(k) \rightarrow \mathbf{Pic}^{5d}_C(k).$$
Since $d\geq 2$, we have $5d\geq 10$. Correspondingly, we get a closed immersion $C\rightarrow Y'$ to a Severi--Brauer variety $Y'$ that's still Brauer equivalent to $X$ since $5$ is coprime to the period of $[X]$ in $\mathrm{Br}(k)$. If $\mathrm{dim}(X)\geq \mathrm{dim}(Y')$, then we are done as there is a closed immersion $Y'\subset X$. Otherwise, by choosing suitable projections we can arrange that there is a closed immersion $C\rightarrow X$ since by assumption we have $\mathrm{dim}(X)>2$. This shows that \ref{prop: dsbd} implies \ref{prop: dci}. The rest is clear.
\end{proof}

The assumption that $\mathrm{dim}(X)\geq 3$ in Proposition \ref{prop: splitandembedd} is optimal in the sense that, given a smooth and proper genus one curve $C$ splitting a Severi--Brauer variety $X$ with $\mathrm{dim}(X)\leq 2$, there may not exist any closed immersion from $C$ to $X$.

If $\mathrm{dim}(X)=1$, then this is obvious. On the other hand, if $\mathrm{dim}(X)=2$ then, by following the above proof, one can show that there always exists a finite map $\varphi:C\rightarrow X$ which induces a birational morphism between $C$ and its image $\varphi(C)$. However, there still may not exist any closed immersion from $C$ to $X$.

\begin{remark}\label{rmk:nonellnorm}
For an explicit example of this latter phenomenon, one can consider any curve $C$ such as those constructed in \cite[Section 3.2]{MR2983078} having period $\mathrm{per}(C)=9$, index $\mathrm{ind}(C)=27$, and which split a Severi--Brauer surface $X$ of index $\mathrm{ind}(X)=3$. If a closed immersion $\varphi:C\rightarrow X$ did exist then, since $\mathrm{dim}(X)=2$, the degree of $\varphi(C)$ must be $\mathrm{deg}(\varphi(C))=3$. Intersecting $\varphi(C)$ with any other curve $\varphi(C)\neq D\subset X$ of degree $\mathrm{deg}(D)=3$ would then produce a Weil divisor on $C$ of degree $9$, which cannot exist when $\mathrm{ind}(C)=27$.

More generally, for any composite integer $n\geq 8$, for any curve $C$ having genus one and with $\mathrm{per}(C)=n$ and $\mathrm{ind}(C)=n^2$, and for any divisor $3<m<n$ of $n$, there exists a Severi--Brauer variety $Y_m$ of dimension $m-1<n-1$ such that $C$ admits a closed immersion to $Y_m$ realizing $C$ as a degree $n(n/m)$ curve in $Y_m$ and such that $C$ does not admit any closed immersion to $Y_m$ as a curve of smaller degree (in particular, there are no Severi--Brauer diagrams involving both $C$ and $Y_m$ even though $C$ does split $Y_m$).

Indeed, since $C$ has $\mathrm{per}(C)=n$, there exists a Severi--Brauer variety $X$ of dimension $n-1$ and a closed immersion from $C$ to $X$ realizing $C$ as a curve of degree $n$ in $X$. Since $\mathrm{ind}(C)=n^2$, the exponent of the central simple algebra $A$ such that $X\cong \mathrm{SB}(A)$ must also be $n$ since, if otherwise, the curve $C$ would contain a Weil divisor of degree strictly smaller than $n^2$. Hence $\mathrm{ind}(A)=n$ as well. Now let $Y_m$ be the Severi--Brauer variety corresponding to the division algebra underlying $A^{\otimes (n/m)}$. One can construct a closed immersion from $C$ to $Y_m$ of degree $n(n/m)$ by composing the twisted Segre embedding \[C\xrightarrow{\Delta} C^{\times (n/m)} \rightarrow X^{\times (n/m)} \rightarrow \mathrm{SB}(A^{\otimes (n/m)})\] with a suitable projection to $Y_m$. Since the exponent of $A^{\otimes (n/m)}$ is exactly $m$ in this case, this embedding from $C$ to $Y_m$ has the minimal degree allowable by the index of $C$.
\end{remark}

\begin{remark} Given a smooth, proper, geometrically connected curve $C$ having genus $g(C)=1$, the collection of Brauer classes split by $C$ form a subgroup of $\mathrm{Br}(k)$ called \textit{the relative Brauer group of $C$ over $k$,} \[\mathrm{Br}(C/k):=\mathrm{ker}\left(\mathrm{Br}(k)\xrightarrow{res_k^{k(C)}} \mathrm{Br}(k(C))\right).\]
If $C(k)\neq\emptyset$, then $\mathrm{Br}(C/k)=0$ since a morphism $C\rightarrow X$ to a Severi--Brauer variety $X$ necessarily induces a $k$-rational point of $X$, implying that $X$ is split itself. Still, even when $C(k)=\emptyset$ so that $C$ is a nontrivial $E=\mathbf{Pic}_C^0$-torsor, such curves $C$ split relatively few Brauer classes. As evidence for this claim, see \cite[Proposition 4.11]{MR3009747} where it's shown that the group $\mathrm{Br}(C/k)$ is finite whenever $k$ is finitely generated over $\mathbb{Q}$. There are even examples for which $\mathrm{Br}(C/k)=0$ even though $C(k)=\emptyset$. Indeed, by the Brauer-Albert-Hasse-Noether theorem, any curve $C$ which defines a nontrivial element of the Tate-Shafarevich group $\Sh(E/k)$ of an elliptic curve $E$ will have $\mathrm{Br}(C/k)=0$ and such examples are classical \cite{MR0041871}.

For more examples and computations: see \cite[Theorem 4.12]{MR3009747} for an example where the relative Brauer group is infinite, see \cite{MR0201435} for the case of a local base field, and see \cite{MR1995537,MR2922613, MR3503968, MR3414415} for examples with curves $C$ of small period.\end{remark}

\subsection{Theta groups and obstruction maps}
In this subsection we first briefly recall the construction of O'Neil's obstruction map, introduced in \cite{MR1924106}, using Mumford's theta groups \cite{MR0204427}. Then, using the twisting language, we show how one can associate to any continuous $1$-cocycle $\xi\in Z^1(G_k,E[n](k^s))$ both a $\xi$-twisted theta group and a $\xi$-twisted obstruction map. We end by relating these objects to their untwisted counterparts.

For any line bundle $\mathcal{L}$ on $E$ of degree $n=\mathrm{deg}(\mathcal{L})\geq 2$, there is an associated theta group $\Theta_{\mathcal{L}}$ which fits into a central extension of group $k$-schemes \begin{equation}\label{eq:centext} 1\to\mathbb{G}_{m}\xrightarrow{\rho}\Theta_{\mathcal{L}}\xrightarrow{\pi} E[n]\to 1. \end{equation} By construction, for any $k$-scheme $S$, the $S$-points of $\Theta_{\mathcal{L}}$ are in natural bijection with the set of pairs $(x,\phi)$ consisting of an $S$-point $x\in E(S)$ and an $S$-isomorphism of geometric vector bundles $\phi:\mathbb{V}(\mathcal{L}^\vee_S)\rightarrow \mathbb{V}(\mathcal{L}^\vee_S)$ which makes the following diagram commutative \[\begin{tikzcd} \mathbb{V}(\mathcal{L}^\vee_S)\arrow["\phi"]{r}\arrow{d} & \mathbb{V}(\mathcal{L}^\vee_S)\arrow{d} \\ E_S\arrow["T_x"]{r} & E_S. \end{tikzcd}\] Here we write $T_x=(p_1,m\circ (x\times\operatorname{id}_{E})):S\times_k E\rightarrow S\times_k E$ for the left translation-by-$x$ map on $E_S:=S\times_k E$, where we write $p_1$ and $p_2$ for the first and second projections from $E_S=S\times_k E$, we write $\mathcal{L}_S$ to mean $p_2^*\mathcal{L}$, and we note that $\phi$ should be $S$-linear on the fibers over any $S$-point $y\in E(S)$.

The group law on $\Theta_{\mathcal{L}}(S)$ is given as composition of diagrams. If $(x,\phi_1)$ and $(y,\phi_2)$ are elements of $\Theta_{\mathcal{L}}(S)$, then we have \[(x,\phi_1)\cdot (y,\phi_2)=(x+y,\phi_1\circ \phi_2).\] The map $\rho$ realizes $\mathbb{G}_m$ as the group of automorphisms of the line bundle $\mathbb{V}(\mathcal{L}^\vee)$, i.e.\ for any $\phi\in \mathbb{G}_m(S)$ we have $\rho(\phi)=(p_S,\phi)$ where $p\in E(k)$ is the identity of $E$. The map $\pi$ in (\ref{eq:centext}) is projection onto the first component, noting an element $(x,\phi)\in \Theta_{\mathcal{L}}(S)$ can exist only if $x\in E[n](S)$. Indeed, if $\phi$ is as above then $\phi$ induces a canonical isomorphism \begin{equation} \label{eq:veciso}\mathbb{V}(\mathcal{L}^\vee_S)\rightarrow E_S\times_{(T_x,E_s)}\mathbb{V}(\mathcal{L}^\vee_S)\cong \mathbb{V}(T_x^*\mathcal{L}^\vee_S)\end{equation} and there exists an isomorphism $\mathcal{L}_S\cong T_x^*\mathcal{L}_S$ only if $x\in E[n](S)$.

A representation of $\Theta_\mathcal{L}$ on the affine $k$-space $\mathrm{H}^0(E,\mathcal{L})$ can be defined by the rule \begin{equation}\label{eq:repmod}\Theta_{\mathcal{L}}\times\mathrm{H}^0(E,\mathcal{L})\rightarrow \mathrm{H}^0(E,\mathcal{L}) \quad \mbox{with}\quad  ((x,\phi), s)\mapsto \phi\circ s\circ T_{-x}\end{equation} where $(x,\phi)\in \Theta_{\mathcal{L}}(S)$ and $s:E_S\rightarrow \mathbb{V}(\mathcal{L}_S^\vee)$ is a section. Equivalently, the corresponding group $k$-scheme homomorphism $\Theta_{\mathcal{L}}\rightarrow \mathbf{GL}(\mathrm{H}^0(E,\mathcal{L}))$ sends $(x,\phi)\in \Theta_{\mathcal{L}}(S)$ to the $S$-point of $\mathbf{GL}(\mathrm{H}^0(E,\mathcal{L}))$ functorially associated to the composition of $\mathcal{O}_S(S)$-modules \begin{equation}\label{eq:rep}\mathrm{H}^0(E_S,\mathcal{L}_S)\xrightarrow{\phi'} \mathrm{H}^0(E_S,T_x^*\mathcal{L}_S)\rightarrow \mathrm{H}^0(E_S,T_{-x,*}\mathcal{L}_S)=\mathrm{H}^0(E_S,\mathcal{L}_S).\end{equation} The leftmost map $\phi'$ of (\ref{eq:rep}) is the map induced by the isomorphism
$\mathcal{L}_S\rightarrow T_x^*\mathcal{L}_S$ from (\ref{eq:veciso}) and the second map of (\ref{eq:rep}) is gotten from adjunction $T_x^*\mathcal{L}_S\rightarrow T_{-x,*}\mathcal{L}_S$ of the canonical isomorphism $T_{-x}^*T_x^*\mathcal{L}_S\cong \mathcal{L}_S$.

With the description of (\ref{eq:repmod}), it's clear that $\mathbb{G}_m\subset \Theta_{\mathcal{L}}$ acts on $\mathrm{H}^0(E,\mathcal{L})$ under the above representation via the canonical scaling action. Hence there is a commutative diagram of group $k$-schemes with exact rows \begin{equation}\label{eq:projrep}\begin{tikzcd}0\arrow{r} & \mathbb{G}_{m}\arrow{r}\arrow[equals]{d} & \Theta_{\mathcal{L}}\arrow{r}\arrow[d] & E[n]\arrow{r}\arrow["\Psi_{\mathcal{L}}"]{d} & 0 \\
      0\arrow{r} & \mathbb{G}_{m}\arrow{r} & \mathbf{GL}(\mathrm{H}^0(E,\mathcal{L}))\arrow{r} & \mathbf{PGL}(\mathrm{H}^0(E,\mathcal{L}))\arrow{r} & 0
   \end{tikzcd}\end{equation}
After choosing a basis for $\mathrm{H}^0(E,\mathcal{L})$, and hence also a morphism $\varphi:E\rightarrow \mathbb{P}(\mathrm{H}^0(E,\mathcal{L}))\cong\mathbb{P}^{n-1}$ representing the complete linear system of $\mathcal{L}$, the rightmost vertical arrow in (\ref{eq:projrep}) has the interpretation that an $S$-point $x\in E[n](S)$ is sent to the unique $S$-automorphism $\tau_x$ of $\mathbb{P}^{n-1}_S$ which makes the diagram \[\begin{tikzcd} E_S\arrow["\varphi_S"]{r}\arrow["T_x"]{d} & \mathbb{P}^{n-1}_S\arrow["\tau_x"]{d}\\
 E_S\arrow["\varphi_S"]{r} & \mathbb{P}^{n-1}_S\end{tikzcd}\] commutative, cf.\ \cite[Proposition 2.1]{MR1924106}. Moreover, by \cite[Theorem 2.5]{MR1825265}, the commutator pairing of the central extension (\ref{eq:centext}), \begin{equation} \label{eq:comm} e^{\mathcal{L}}:E[n](S)\times E[n](S)\rightarrow \mathbb{G}_m(S) \quad\mbox{defined by} \quad e^{\mathcal{L}}(x,y)=[\tilde{x},\tilde{y}]:=\tilde{x}\tilde{y}\tilde{x}^{-1}\tilde{y}^{-1}\end{equation} where $\tilde{x},\tilde{y}\in \Theta_{\mathcal{L}}(S)$ are arbitrary lifts of $x,y\in E[n](S)$ respectively, coincides with the Weil pairing on $E[n]$ so long as $\mathrm{char}(k)\nmid n$.

\begin{definition}
    \label{def:obstructionMap}
If $\mathcal{L}$ is a line bundle of degree $n\geq 2$ on $E$, then we define the \textit{obstruction map corresponding to $\mathcal{L}$} to be the connecting map in Galois cohomology \[
    \Delta_{\mathcal{L}}:\mathrm{H}^{1}(G_{k},E[n](k^s))\to \mathrm{H}^{2}(G_{k},(k^s)^\times)
  \] associated to the central extension (\ref{eq:centext}) of $\Theta_{\mathcal{L}}$.
\end{definition}

Among all possible obstruction maps coming from line bundles $\mathcal{L}$ on $E$, we specify a preferred one. Namely, if $\mathcal{L}=\mathcal{O}(p)^{\otimes n}$. In this case we write $\Theta_n$ for $\Theta_{\mathcal{L}}$, $\Delta_n$ for $\Delta_{\mathcal{L}}$, and we call $\Delta_n$ \textit{the degree-$n$ obstruction map}. The map $\Delta_n$ has the following interpretation:

\begin{proposition}[{\cite[Section 2]{MR2384334}}] \label{prop:obsforget}
Write $\pi_0(\mathsf{SB}_n)(k)$ for the set of isomorphism classes of Severi--Brauer varieties of dimension $n-1$ over $k$ and pointed by the isomorphism class $[\mathbb{P}^{n-1}_k]$. Then there is a commuting square \[\begin{tikzcd}\pi_0(E-\mathsf{SB}_n)(k)\arrow[equals]{r}\arrow{d} & \mathrm{H}^1(G_k,E[n](k^s))\arrow["\Psi_*"]{d} \\ \pi_0(\mathsf{SB}_n)(k)\arrow[equals]{r} & \mathrm{H}^1(G_k,\mathbf{PGL}_n(k^s))\end{tikzcd}\] where the left vertical arrow sends the class of a degree-$n$ Severi--Brauer diagram $[(C,\,\mu,\,f:C\rightarrow \mathrm{SB}(A))]$ to $[\mathrm{SB}(A)]$, the right vertical $\Psi_*$ is the pushforward along the map of (\ref{eq:projrep}) when $\mathcal{L}=\mathcal{O}(p)^{\otimes n}$, and the bottom horizontal map is constructed by considering Severi--Brauer varieties of dimension $n-1$ as twists of $\mathbb{P}^{n-1}_k$.

Composing with the connecting map of the bottom row of (\ref{eq:projrep}) then gives a commuting square
\[\begin{tikzcd}\pi_0(E-\mathsf{SB}_n)(k)\arrow[equals]{r}\arrow{d} & \mathrm{H}^1(G_k,E[n](k^s))\arrow["\Delta_n"]{d} \\ \mathrm{Br}(k)\arrow["\cong"]{r} & \mathrm{H}^2(G_k,(k^s)^\times)\end{tikzcd}\] where the left vertical arrow is similarly given by sending $[(C,\,\mu,\,f:C\rightarrow \mathrm{SB}(A))]$ to $[\mathrm{SB}(A)]\in \mathrm{Br}(k)$.$\hfill\square$
\end{proposition}

\begin{corollary}
Let $k$ be a perfect field, let $E/k$ be an elliptic curve, and let $C$ be a curve with $\mathbf{Pic}^0_C\cong E$. For any integer $n\geq 2$, let $S_n(C)\subset \mathrm{H}^1(G_k,E[n](k^s))$ denote the set of elements corresponding to degree-$n$ Severi-Brauer diagrams under $E$ which include the curve $C$ in their definition.
Then $\mathrm{Br}(C/k)=\bigcup_{n\geq 2} \Delta_n(S_n(C))$.
\end{corollary}

\begin{proof}
Clearly any Brauer class contained in the union $\bigcup_{n\geq 2} \Delta_n(S_n(C))$ is split by $C$ and, hence, is contained in $\mathrm{Br}(C/k)$. Conversely, if $C$ splits a Brauer class $\alpha\in \mathrm{Br}(k)$ then by the proof of Proposition \ref{prop: splitandembedd} there is a Severi--Brauer variety $Y$ of dimension $d-1$ (for some $d$) and a Severi--Brauer diagram $(C,\mu, f:C\rightarrow Y)$ of degree $d$ under $E$ giving a class in $\xi\in \mathrm{H}^1(G_k,E[d](k^s))$ with $\alpha=\Delta_d(\xi)$ by Proposition \ref{prop:severiBrauerDiagramsBijections}.
\end{proof}

\begin{remark}\label{rmk:quadratic}
The obstruction maps $\Delta_{\mathcal{L}}$, having been defined via connecting maps in nonabelian group cohomology, are only pointed maps and not group homomorphisms.
In fact, if $\mathcal{L}$ is symmetric in the sense of Mumford \cite[Section 2]{MR0204427}, i.e.\ $\mathcal{L}\cong i^*\mathcal{L}$ under the involution $i:E\rightarrow E$ defined by $i(x)=-x$, then the map $\Delta_{\mathcal{L}}$ has the property that for any integer $m\in \mathbb{Z}$ and for any $a\in \mathrm{H}^1(G_k,E[n](k^s))$ one has $\Delta_{\mathcal{L}}(ma)=m^2\Delta_{\mathcal{L}}(a)$ (see \cite[\S2]{Zarkhin} or use the diagram $(*)_n$ in \cite[p.\ 308]{MR0204427}). 

Moreover, Zarkhin shows in \cite[\S1, Theorem]{Zarkhin} that the map $\Delta_{\mathcal{L}}$ is related to the commutator pairing $e^{\mathcal{L}}$ from (\ref{eq:comm}) in the following way. The composition \[ \mathrm{H}^1(G_k,E[n](k^s))\times \mathrm{H}^1(G_k,E[n](k^s))\xrightarrow{\cup} \mathrm{H}^2(G_k,E[n](k^s)\otimes E[n](k^s))\xrightarrow{e^{\mathcal{L}}_*} \mathrm{H}^2(G_k, (k^s)^\times),\] of the cup product and induced pairing map, is symmetric and bilinear. The obstruction map $\Delta_{\mathcal{L}}$ relates to this composition by the formula \begin{equation} \label{eq:bilinear} e^{\mathcal{L}}_*(a\cup b)=\Delta_{\mathcal{L}}(a+b)-\Delta_{\mathcal{L}}(a)-\Delta_{\mathcal{L}}(b).\end{equation} Combined with the quadratic property above, this gives the relation $e^{\mathcal{L}}_*(a\cup a)=2\Delta_{\mathcal{L}}(a)$. In the case that $n\in \mathbb{Z}$ is odd, this also gives $\Delta_{\mathcal{L}}(a)=\frac{1}{2}e^{\mathcal{L}}_*(a\cup a)$. From now on, we will write $e^n_*$ when $\mathcal{L}=\mathcal{O}(p)^{\otimes n}$.
\end{remark}

For any given degree-$n$ line bundle $\mathcal{L}$ on $E$, the obstruction map $\Delta_{\mathcal{L}}$ for $\mathcal{L}$ can be determined from $\Delta_n$. More generally, if we consider any continuous $1$-cocycle $\xi\in Z^1(G_k,E[n](k^s))$ whose class in $\mathrm{H}^1(G_k,E[n](k^s))$ defines the class of a degree-$n$ Severi--Brauer diagram $T([\xi])=[(C,\mu,f:C\rightarrow \mathrm{SB}(A))]$ then, by twisting the diagram (\ref{eq:projrep}) with $\mathcal{L}=\mathcal{O}(p)^{\otimes n}$, one gets a commutative ladder of group $k$-schemes with exact rows \begin{equation}\label{eq:projreptwist}\begin{tikzcd}0\arrow{r} & \mathbb{G}_{m}\arrow{r}\arrow[equals]{d} & \Theta_{\xi}\arrow{r}\arrow[d] & E[n]\arrow{r}\arrow["\Psi_{\xi}"]{d} & 0 \\
      0\arrow{r} & \mathbb{G}_{m}\arrow{r} & \mathbf{GL}(A)\arrow{r} & \mathbf{PGL}(A)\arrow{r} & 0.
   \end{tikzcd}\end{equation}
To be precise, one works with the diagram (\ref{eq:projrep}) scalar extended to the separable closure $k^s$ of $k$ and defines a twisted action of $G_k$ on each group $k$-scheme by:
\begin{itemize}
\item $\sigma\in G_k$ acts on $E[n]_{k^s}$ via the conjugation $\xi_\sigma (\sigma(-))\xi_\sigma^{-1}$; since $E[n]$ is commutative, this action reduces to the usual action of $G_k$ on $E[n]_{k^s}$.
\item By Hilbert's theorem 90, for each $\sigma\in G_k$ one can choose a lift $\tilde{\xi}_{\sigma}\in \Theta_n(k^s)$ of $\xi_\sigma\in E[n](k^s)$; now $\sigma$ acts on $(\Theta_n)_{k^s}$ via the conjugation $\tilde{\xi}_{\sigma}(\sigma(-))\tilde{\xi}_{\sigma}^{-1}$. Since $\mathbb{G}_m$ is central in $\Theta_n$, this action is independent of the choice of lifts.
\item $\sigma\in G_k$ acts on $\mathbf{GL}_n$ and on $\mathbf{PGL}_n$ by conjugation as well, using the elements induced from $\xi_\sigma$ and $\tilde{\xi}_\sigma$ under the compositions $\Theta_n\rightarrow \mathbf{GL}_n$ and $E[n]\rightarrow \mathbf{PGL}_n$.
\end{itemize}
Galois descent then guarantees the existence of the diagram (\ref{eq:projreptwist}) defined over $k$.

\begin{definition} For any continuous $1$-cocycle $\xi\in Z^1(G_k,E[n](k^s))$, we call $\Theta_{\xi}$ the \textit{$\xi$-twisted Theta group associated to $E[n]$} and we write \[\Delta_\xi: \mathrm{H}^1(G_k,E[n](k^s))\rightarrow \mathrm{H}^2(G_k,(k^s)^\times)\] for the associated \textit{$\xi$-twisted obstruction map} gotten as the connecting map in Galois cohomology associated to the central extension of the top row in (\ref{eq:projreptwist}). 
\end{definition}

\begin{remark}
The group $k$-scheme $\Theta_\xi$ is determined, up to isomorphism, by the class $[\xi]\in \mathrm{H}^1(G_k,E[n](k^s))$. The map $\Delta_\xi$ also depends on just the class $[\xi]\in \mathrm{H}^1(G_k,E[n](k^s))$. However, the procedure to construct the twisted diagram (\ref{eq:projreptwist}) does depend on the choice of $1$-cocycle $\xi\in Z^1(G_k,E[n](k^s))$.
\end{remark}

\begin{remark}\label{rmk:linexi}
Let $\mathcal{L}$ be any fixed line bundle of degree $n$ on $E$. Let $\xi\in Z^1(G_k,E[n](k^s))$ be a continuous $1$-cocycle gotten from an explicit choice of splitting of the Severi--Brauer diagram $(E,m, \varphi:E\rightarrow \mathbb{P}^{n-1})$ where $\varphi$ is associated to the complete linear system of $\mathcal{L}$. Then it's shown in \cite{MR2384334}, in the exposition immediately after Proposition 1.31, that the $\xi$-twisted diagram (\ref{eq:projreptwist}) is, up to a change of basis for $\mathrm{H}^0(E,\mathcal{L})$, the same as the diagram (\ref{eq:projrep}) associated to the line bundle $\mathcal{L}$. In particular, $\Delta_\mathcal{L}=\Delta_\xi$.
\end{remark}

\begin{proposition}\label{prop:twistob}
Let $\xi\in Z^1(G_k,E[n](k^s))$ be a continuous $1$-cocycle. The following diagram commutes \[\begin{tikzcd}\mathrm{H}^1(G_k,E[n](k^s))\arrow["+{\left[\xi\right]}"]{r}\arrow["\Delta_\xi"]{d} & \mathrm{H}^1(G_k,E[n](k^s))\arrow["\Delta_n"]{d} \\ \mathrm{H}^2(G_k,(k^s)^\times)\arrow["+\Delta_n({[\xi]})"]{r} & \mathrm{H}^2(G_k,(k^s)^\times).\end{tikzcd}\] Hence $\Delta_\xi(a)=\Delta_n(a+[\xi])-\Delta_n([\xi])$ for all $a\in \mathrm{H}^1(G_k,E[n](k^s))$.
\end{proposition}

\begin{proof}
This is a formal consequence of the twisting theory, see \cite[Proposition (28.12)]{MR1632779}.
\end{proof}

\begin{corollary}\label{cor: splittoo}
Let $E/k$ be an elliptic curve and let $\mathcal{L}$ be a degree-$n$ line bundle on $E$. If there exists an element $a\in \mathrm{H}^1(G_k,E[n](k^s))$ with $\Delta_{\mathcal{L}}(a)=\alpha\in \mathrm{Br}(k)$, then there exists a principal homogeneous space $C$ for $E$ which splits $\alpha$. Further, there is some $z\in \mathrm{H}^1(G_k,E[n](k^s))$ such that  $\alpha=\Delta_n(z)$. 
\end{corollary}

\begin{proof}
Let $\xi:G_k\rightarrow E[n](k^s)$ be any continuous $1$-cocycle gotten from the Severi--Brauer diagram $(E,m,\varphi)$ as in Remark \ref{rmk:linexi}. Then from Proposition \ref{prop:twistob} we have \[ \alpha=\Delta_{\mathcal{L}}(a)=\Delta_n(a+[\xi])-\Delta_n([\xi]).\] From Proposition \ref{prop:obsforget}, we have $\Delta_n([\xi])=0$. Hence $\Delta_n(a+[\xi])=\alpha$ and the claim follows by applying Proposition \ref{prop:obsforget} again.
\end{proof}

\begin{remark}
Another way to view Proposition \ref{prop:twistob} is as follows. Let $\xi=(C,\,\,\mu,\,\,f:C\rightarrow \mathrm{SB}(A))$ be a fixed degree-$n$ Severi--Brauer diagram for $E$. Let $\pi_0(E-\mathsf{SB}_n)(k)_\xi$ be the set of isomorphism classes of degree-$n$ Severi--Brauer diagrams over $k$ and under $E$, considered as twists of $\xi$, and let $\pi_0(\mathsf{SB}_n)(k)_\xi$ be the set of isomorphism classes of Severi--Brauer varieties over $k$ of dimension $n-1$, considered as twists of $\mathrm{SB}(A)$. Then, assuming either \ref{item: f1} or \ref{item: f2} of Proposition \ref{prop:severiBrauerDiagramsBijections}, there is a commuting diagram (see \cite[Functorality, p.\ 387]{MR1632779} for the front-most square):
	\[\begin{tikzcd}[row sep=scriptsize, column sep=scriptsize]
		& \pi_0(E-\mathsf{SB}_n)(k)_\xi \arrow[swap, "S_\xi", dl] \arrow[equals, rr] \arrow[dd] & & \pi_0(E-\mathsf{SB}_n)(k) \arrow["S", dl] \arrow[dd] \\
		\mathrm{H}^1(G_k,E[n](k^s)) \arrow["+{[\xi]}", near end, rr, crossing over] \arrow["\Psi_{\xi*}",dd] & & \mathrm{H}^1(G_k,E[n](k^s)) \\
		& \pi_0(\mathsf{SB}_n)(k)_\xi \arrow[dl] \arrow[equals, rr] & & \pi_0(\mathsf{SB}_n)(k) \arrow[dl] \\
		\mathrm{H}^1(G_k,\mathbf{PGL}(A)(k^s)) \arrow[rr] & & \mathrm{H}^1(G_k,\mathbf{PGL}_n(k^s)). \arrow["\Psi_*", near start, from=uu, crossing over]\\
	\end{tikzcd}\]
 The vertical arrows in this diagram can all be viewed as ``forgetful" maps. The horizontal arrows can be considered as taking an object, considered as a twist of $\xi$ (or of $\mathrm{SB}(A)$), and considering it instead as a twist of the trivial degree-$n$ Severi--Brauer diagram under $E$ (or as a twist of $\mathbb{P}^{n-1}$ respectively).
\end{remark}
  
\section{Compatibility with the Galois symbol}\label{sec:symbols}
We continue to work over a fixed but arbitrary base field $k$. In this section, we choose an integer $n\geq 2$ which is indivisible by the characteristic of the field $k$ and we assume that $E$ is an elliptic curve defined over $k$ admitting an isomorphism of group schemes $\phi_n: E[n]\cong \mathbb{Z}/n\mathbb{Z}\times \mu_n$. Such an isomorphism $\phi_n$ is called a \textit{full level-$n$ structure on $E[n]$}. 

Let $t\in E[n](k)$ be such that $\langle \phi(t)\rangle=\mathbb{Z}/n\mathbb{Z}\times 1$. Write $\mathcal{L}=\mathcal{O}(t+2t+\cdots + (n-1)t+nt)$ for the degree $n$ line bundle on $E$ gotten by summing all $n$ prime divisors corresponding to points in the cyclic subgroup of $E[n](k)$ generated by $t$. For simplicity of notation, we write $\Delta_t$ for the obstruction map of Definition \ref{def:obstructionMap} associated to the bundle $\mathcal{L}$. Note both $\mathcal{L}$ and $\Delta_t$ are independent of the choice of generator for $\langle t\rangle \subset E[n](k)$.

\begin{remark}\label{rmk:deltatandn}
We can relate the the obstruction map $\Delta_t$ to $\Delta_n$ using Proposition \ref{prop:twistob}. Using Remark \ref{rmk:Kummer}, the two obstruction maps are equal precisely when there exists an isomorphism $\mathcal{L}\cong \mathcal{O}(q)^{\otimes n}$ for some point $q\in E(k)$. If $n$ is odd, then one can check $\mathcal{L}\cong \mathcal{O}(p)^{\otimes n}$ where $p\in E(k)$ is the identity. If $n$ is even, then since \begin{align*}\mathcal{O}(t+2t+\cdots + nt)&\cong \mathcal{O}((t-p) + (2t- p) + \cdots + (nt-p))\otimes \mathcal{O}(p)^{\otimes n}\\ 
&\cong \mathcal{O}\left(\frac{n(n-1)}{2}\left(t-p\right)\right)\otimes \mathcal{O}(p)^{\otimes n}\\
&\cong \mathcal{O}\left(\frac{n}{2}t\right)\otimes \mathcal{O}(p)^{\otimes n-1}
\end{align*} we find that $\mathcal{L}\cong \mathcal{O}(q)^{\otimes n}$ for some point $q\in E(k)$ whenever there exists a point $q\in E(k)$ such that $\frac{n}{2}t=nq$. If there is a point $q\in E[2n](k)$ with $2q=t$, then this latter equality certainly holds. Hence, if $n$ is odd or if $n$ is even and there exists $q\in E[2n](k)$ with $2q=t$ then $\Delta_t=\Delta_n$.
\end{remark}

Antieau and Auel \cite[Proposition 2.13 and Proposition 2.14]{https://doi.org/10.48550/arxiv.2106.04291} have given an explicit description for the obstruction map $\Delta_{t}$ in terms of cyclic algebras; their proof uses a characterization of this functorial transformation as a Brauer class on the classifying stack of $E[n]$. However, in the case when $k$ contains a primitive $n$th root of unity, this description had been obtained earlier by Clark and Sharif \cite[Theorem 10 and Lemma 11]{MR2592017} using Galois cohomology. Here we generalize the proof of Clark and Sharif, still using Galois cohomology, in order to deduce the result of Antieau and Auel; by doing so, we're also able to simplify the description of \cite[Proposition 2.14]{https://doi.org/10.48550/arxiv.2106.04291} in the case that $E[2n]$ has full level-$2n$ structure.

\subsection{Galois symbols and cyclic algebras}\label{ss:cyclicalg}
For an arbitrary integer $n\geq 1$, for any Galois $(\mathbb{Z}/n\mathbb{Z})$-algebra $L$ over $k$, and for any unit $a\in k^\times$ one can construct the \textit{cyclic central simple $k$-algebra} $(L,a)$. By construction, the algebra $(L,a)$ is generated by $L$ and a single element $z$ with multiplication $z^n=a$ and $zb=\rho(b)z$ for any $b\in L$ and with $\rho=1\in \mathbb{Z}/n\mathbb{Z}$. When $k$ contains a primitive $n$th root of unity, we can identify $L\cong k(\sqrt[n]{c})^{\oplus d}$ for some $c\in k^\times$ with exact order $n/d$ in the group $k^\times/k^{\times n}$ for some divisor $d$ of $n$. There is then a generator $\sigma\in \mathrm{Gal}(k(\sqrt[n]{c})/k)$ with $\sigma(\sqrt[n]{c})=\zeta_{n/d}\sqrt[n]{c}$, for some primitive $(n/d)$th root of unity $\zeta_{n/d}$, such that the $\mathbb{Z}/n\mathbb{Z}$-algebra structure on $L$ is given by \[\rho(x_1,...,x_d)=(\sigma(x_d),x_1,...,x_{d-1})\] for $\rho=1\in \mathbb{Z}/n\mathbb{Z}$. In this case we'll write $(c,a)_{n}$ for the algebra $(L,a)$ and call it the \textit{degree-$n$ symbol algebra} associated to the pair $a,c\in k^\times$.

One can give an alternative, cohomological description of the cyclic central simple $k$-algebra $(L,a)$ when $n$ is indivisible by the characteristic of the base field $k$ as follows. First, one can identify $a$ with its class in $k^\times/k^{\times n}\cong \mathrm{H}^1(G_k,\mu_n(k^s))$ and $L$ with a character $\chi:G_k\rightarrow \mathbb{Z}/n\mathbb{Z}$ via \cite[(28.15) Example]{MR1632779}. Taking the cup product of the associated classes $\chi\in \mathrm{H}^1(G_k,\mathbb{Z}/n\mathbb{Z})$ and $[a]\in \mathrm{H}^1(G_k,\mu_n(k^s))$ yields a class in $\mathrm{H}^2(G_k,\mu_n(k^s))\subset \mathrm{H}^2(G_k,(k^s)^\times)$. If one identifies the latter group with $\mathrm{Br}(k)$ via the crossed-product construction, then there is an equality
\[\chi\cup [a]=[(L,a)]\in \mathrm{Br}(k).\]
We recommend \cite[\S30.A]{MR1632779} for proofs of these facts and more.

\begin{definition}
We define the \textit{degree-$n$ Galois symbol} associated to a character $\chi\in \mathrm{H}^1(G_k,\mathbb{Z}/n\mathbb{Z})$ and $a\in k^\times$ to be the cup product $\chi\cup [a]\in \mathrm{H}^2(G_k,(k^s)^\times)$. We'll write $[\chi,a)_n$ for this symbol.
\end{definition}

\begin{remark}
Identifying $\mathrm{H}^2(G_k,(k^s)^\times)$ with $\mathrm{Br}(k)$ via the crossed product construction is the negative of the inverse of the map used in Proposition \ref{prop:obsforget}.
\end{remark}

Galois symbols of varying degrees are often related in several ways. One such way that we will use later is the following:

\begin{lemma}\label{lem:algebra}
Let $m\geq 1$ be an integer indivisible by the characteristic of $k$. Let $\chi\in \mathrm{H}^1(G_k,\mathbb{Z}/n\mathbb{Z})$ correspond to the Galois $(\mathbb{Z}/n\mathbb{Z})$-algebra $L$ over $k$, and let $b\in k^\times$ be any element. Then there exists a natural Galois $(\mathbb{Z}/nm\mathbb{Z})$-algebra structure on $L^{\times m}=L\times \cdots \times L$ corresponding to the character $m\circ\chi\in \mathrm{H}^1(G_k,\mathbb{Z}/nm\mathbb{Z})$, where $m:\mathbb{Z}/n\mathbb{Z}\rightarrow \mathbb{Z}/nm\mathbb{Z}$ is the inclusion-by-$m$ map, and an equality \[[\chi,b)_n=[m\circ \chi,b)_{nm}\] inside $\mathrm{H}^2(G_k,(k^s)^\times)$.
\end{lemma}

\begin{proof}
Let $\sigma=1\in \mathbb{Z}/nm\mathbb{Z}$. Then the action of $\sigma$ on $L^{\times m}$ is defined as follows \[\sigma(a_1,...,a_m)=(\bar{\sigma}(a_m), a_1, ..., a_{m-1})\] where $\bar{\sigma}=1\in \mathbb{Z}/n\mathbb{Z}$. One checks that this makes $L^{\times m}$ a Galois $(\mathbb{Z}/nm\mathbb{Z})$-algebra and hence it corresponds to a character $\chi'\in \mathrm{H}^1(G_k,\mathbb{Z}/nm\mathbb{Z})$. Following \cite[(28.15) Example]{MR1632779}, one can also check that this character is $\chi'= m\circ \chi$.

To see that $[\chi,b)_n=[\chi',b)_{nm}$, we will show that the associated algebras in $\mathrm{Br}(k)$ are Brauer equivalent via the crossed-product map. Let $(L^{\times m},b)$ be generated by $L^{\times m}$ and an element $z$ such that $z^{nm}=b$ and let $(L,b)$ be generated by $L$ and an element $w$ with $w^n=b$. Then there is an isomorphism \[\varphi:(L^{\times m},b)\cong M_m((L,b))\] defined on $L^{\times m}$ as \[\varphi(\vec{a})=\mathrm{diag}(a_1,...,a_m) \quad\quad \mbox{for all } \vec{a}=(a_1,...,a_m)\in L^{\times m}\] and on the element $z$ as \[ \varphi(z)=\begin{pmatrix}0 & \cdots & 0 & w \\ 1 & \cdots & 0 & 0\\ \vdots & \ddots & \vdots & \vdots \\ 0 & \hdots & 1 & 0 \end{pmatrix}.\] One can check directly that $\varphi(z)^{nm}=b\cdot \mathrm{Id}_{m\times m}$ and $\varphi(\sigma(\vec{a}))\varphi(z)=\varphi(z)\varphi(\vec{a})$, so the map is well-defined. The map $\varphi$ is therefore an isomorphism for dimension reasons.
\end{proof}

\subsection{The obstruction map \texorpdfstring{$\Delta_t$}{Deltat}}
Now assume that $n\geq 2$ is an integer indivisible by the characteristic of the base field $k$. Suppose also that we have an isomorphism of group schemes $\phi_n:E[n]\cong \mathbb{Z}/n\mathbb{Z}\times \mu_n$. We fix throughout this subsection a primitive $n$th root of unity $\zeta_n\in k^s$. There are then points $t,s\in E[n](k^s)$ such that $\phi_n(t)=(1,1)$ and $\phi_n(s)=(0,\zeta_n)$. Let $\mathcal{L}=\mathcal{O}(t+\cdots + nt)$ be the associated line bundle and write $\zeta_n^c=e^{\mathcal{L}}(t,s)$ for the primitive $n$th root of unity gotten from the Weil pairing of $E[n]$ by (\ref{eq:comm}). Note that we will have $c\not\equiv 1 \pmod{n}$ in general.

\begin{lemma}\label{lem: delta}
Let $\delta_s$ be the composition of $\iota_*$ and $\Delta_t$, \[ \delta_s:=\Delta_t\circ \iota_*:\mathrm{H}^1(G_k,\mu_n(k^s))\xrightarrow{\iota_*} \mathrm{H}^1(G_k,E[n](k^s))\xrightarrow{\Delta_t} \mathrm{H}^2(G_k,(k^s)^\times)\] where $\iota$ is the composition \[\iota:\mu_n\simeq 0\times \mu_n\xrightarrow{\phi_n^{-1}} E[n].\] 
Then the following statements hold:
\begin{enumerate}[label=(R\arabic*)]
\item\label{lem:natres} $\delta_s$ is natural with respect to restrictions along separable subfields $k^s\supset F \supset k$.
\item\label{lem:group} $\delta_s$ is a group homomorphism.
\item\label{lem:2tor} For any element $b\in \mathrm{H}^1(G_k,\mu_n(k^s))$ we have $2\delta_s(b)=0$.
\end{enumerate}
\end{lemma}

\begin{proof}
The meaning of \ref{lem:natres} is that $\delta_s$ and the analogously defined map defined over a separable extension $F/k$ contained in $k^s$ forms a commutative square with respect to restrictions. This follows from the fact that the inclusion $\mathrm{H}^1(G_k,\mu_n(k^s))\rightarrow \mathrm{H}^1(G_k,E[n](k^s))$ is induced by a morphism of group schemes together with the same compatibility statement for $\Delta_t$. For $\Delta_t$, this compatibility similarly follows from the fact that the central extension (\ref{eq:centext}) is an extension of group schemes.

For \ref{lem:group}, we can note by Remark \ref{rmk:quadratic} that \[e^\mathcal{L}_*(\iota_*(a)\cup \iota_*(b))=\delta_s(a+b)-\delta_s(a)-\delta_s(b)\] for any $a,b\in \mathrm{H}^1(G_k,\mu_n(k^s))$. However, since the Weil pairing is alternating we have $e^\mathcal{L}_*(\iota_*(a)\cup\iota_*(b))=0$.

Finally, \ref{lem:2tor} follows from the fact that $\delta_s$ is both a homomorphism and quadratic: 
\[4\delta_s(b)=\delta_s(2b)=\delta_s(b+b)=\delta_s(b)+\delta_s(b)\] which holds for any $b\in \mathrm{H}^1(G_k,\mu_n(k^s))$.
\end{proof}

\begin{theorem}\label{thm:antieauauel} Let $e^\mathcal{L}(\bullet,-):E[n]\rightarrow E[n]^\vee$ denote the isomorphism of group schemes induced by the Weil pairing (\ref{eq:comm}) which assigns to a point $z\in E[n](k^s)$ the functional $e^\mathcal{L}(z,-):E[n](k^s)\rightarrow \mu_n(k^s)$. Denote by $ev\circ (\phi_n^{-1})^\vee:E[n]^\vee\cong \mu_n\times \mathbb{Z}/n\mathbb{Z}$ the isomorphism of group schemes composing the dual of $\phi_n^{-1}$ with the evaluation at $\phi_n(t)$ and $\mathrm{log}_{\zeta_n^c}$ of the evaluation at $\phi_n(s)$. Then the obstruction map $\Delta_t$ fits into the following commuting diagram.
\begin{equation}
\begin{tikzcd}
\mathrm{H}^1(G_k,E[n](k^s))\arrow["e^{\mathcal{L}}{(\bullet,-)}_{*}"]{r}\arrow["{\Psi_{\mathcal{L}}}_*"]{d} & \mathrm{H}^1(G_k,E[n]^\vee(k^s))\arrow["ev\circ (\phi_n^{-1})^\vee_*"]{r} & \mathrm{H}^1(G_k,\mu_n(k^s))\times \mathrm{H}^1(G_k,\mathbb{Z}/n\mathbb{Z})\arrow["{(\mathrm{id}_*,\, \cup)}"]{d} \\ \mathrm{H}^1(G_k,\mathrm{PGL}_n(k^s))\arrow["\partial"]{dr} & & \mathrm{H}^1(G_k,\mu_n(k^s))\times \mathrm{H}^2(G_k,\mu_n(k^s))\arrow["{(b,a)\mapsto a+\delta_s(b)}"]{dl} \\ & \mathrm{H}^2(G_k,(k^s)^\times)\hspace{-0.5 em} & 
\end{tikzcd}
\end{equation}
Equivalently, if for an element $\xi\in \mathrm{H}^1(G_k,E[n](k^s))$ the push-forward $(\phi_n)_*(\xi)$ corresponds to a character $\chi\in\mathrm{H}^1(G_k,\mathbb{Z}/n\mathbb{Z})$ and a class $[b]\in k^\times/k^{\times n}\cong\mathrm{H}^1(G_k,\mu_n(k^s))$ then we have \[\Delta_t([\xi])=c[\chi,b)_n+\delta_s([b])\] where $c\in \mathbb{Z}$ is such that $\zeta_n^c=e^{\mathcal{L}}(t,s)$.
\end{theorem}

\begin{proof}
We can check this directly at the level of cocycles. To start, let $\xi:G_k\rightarrow E[n](k^s)$ be any continuous $1$-cocycle with class $[\xi]\in \mathrm{H}^1(G_k,E[n](k^s))$. We'll write $M_t,M_s\in \mathrm{GL}_n(k^s)$ for lifts of the elements $\Psi_{\mathcal{L}}(t)$ and $\Psi_{\mathcal{L}}(s)$, respectively, under the map $\Psi_{\mathcal{L}}:E[n](k^s)\rightarrow \mathrm{PGL}_n(k^s)$. Note that, by Hilbert's theorem 90 and because $t\in E[n](k)$ is defined $k$-rationally, we can choose $M_t\in \mathrm{GL}_n(k)$. Further, by \cite[Corollary 2.8]{MR1825265} we can arrange so that $M_t^n=\mathrm{Id}_{n\times n}$. We then define lifts of the elements $at+bs\in E[n](k^s)$ for any pair of integers $0\leq a,b <n$ by setting $M_{at+bs}:=M_t^a\cdot M_s^b$.

For any $\sigma \in G_k$, write $\xi(\sigma)=a_\sigma t+ b_\sigma s$ for integers $0\leq a_{\sigma},b_{\sigma}<n$. Under the isomorphism \[(\phi_n)_*:\mathrm{H}^1(G_k,E[n](k^s))\xrightarrow{\sim} \mathrm{H}^1(G_k,\mathbb{Z}/n\mathbb{Z})\times \mathrm{H}^1(G_k,\mu_n(k^s))\] we have $(\phi_n)_*([\xi])=([\alpha],[\beta])$ where $\alpha(\sigma)=a_\sigma$ and $\beta(\sigma)=\zeta_n^{b_{\sigma}}$. Following the definition of $\partial\circ {\Psi_{\mathcal{L}}}_*$ applied to $\xi$ then yields the following $2$-cocycle
\begin{align*} \partial\circ {\Psi_{\mathcal{L}}}_*(\xi)(\sigma,\tau) &= M_{\xi(\sigma)}\cdot \sigma(M_{\xi(\tau)}) \cdot (M_{\xi(\sigma\tau)})^{-1}\\ &= (M_t^{a_{\sigma}} M_s^{b_{\sigma}}) \cdot \sigma(M_{t}^{a_\tau}\cdot M_s^{b_\tau}) \cdot (M_s^{-b_{\sigma\tau}} M_t^{-a_{\sigma\tau}})\\ &= (M_t^{a_{\sigma}} M_s^{b_{\sigma}}) \cdot M_{t}^{a_\tau} \sigma(M_s^{b_\tau}) \cdot (M_s^{-b_{\sigma\tau}} M_t^{-a_{\sigma\tau}}) \end{align*} for any $\sigma,\tau \in G_k$.

Since both $M_{s}^{b_{\sigma\tau}}$ and $M_{s}^{b_\sigma}\cdot \sigma(M_s^{b_{\tau}})$ lift $\Psi_{\mathcal{L}}(\phi_n^{-1}(\beta(\sigma\tau)))\in \mathrm{PGL}_n(k^s)$, there exists a constant $d_{\sigma,\tau}\in (k^s)^\times$ such that \[M_{s}^{b_{\sigma\tau}}=d_{\sigma,\tau}\cdot M_{s}^{b_\sigma}\sigma(M_s^{b_{\tau}}).\] Note that this also implies $(1/d_{\sigma,\tau})=\partial\circ{\Psi_{\mathcal{L}}}_*((\phi_n)_*^{-1}(0,[\beta]))(\sigma,\tau)$. Moreover, by the commutativity of (\ref{eq:projrep}), we get \begin{equation} \label{eq:deltahereitis}(1/d_{\sigma,\tau})=\partial\circ{\Psi_{\mathcal{L}}}_*((\phi_n)_*^{-1}(0,[\beta]))(\sigma,\tau)=\delta_s([\beta])(\sigma,\tau).\end{equation} Plugging this into the above gives
\begin{align*} \partial\circ {\Psi_{\mathcal{L}}}_*(\xi)(\sigma,\tau) &= (M_t^{a_{\sigma}} M_s^{b_{\sigma}}) \cdot M_{t}^{a_\tau} \sigma(M_s^{b_\tau}) \cdot (M_s^{-b_{\sigma\tau}} M_t^{-a_{\sigma\tau}})\\
&= (1/d_{\sigma,\tau})\cdot (M_t^{a_{\sigma}} M_s^{b_{\sigma}}) \cdot M_{t}^{a_\tau} \sigma(M_s^{b_\tau})\cdot (\sigma(M_s^{-b_{\tau}})\cdot M_s^{-b_{\sigma}}M_t^{-a_{\sigma\tau}})\\
&=(1/d_{\sigma,\tau})\cdot M_t^{a_{\sigma}} M_s^{b_{\sigma}} \cdot M_{t}^{a_\tau}\cdot M_s^{-b_{\sigma}}M_t^{-a_{\sigma\tau}}\\
&=(1/d_{\sigma,\tau})\cdot M_t^{a_{\sigma}}\cdot (M_t^{a_{\tau}}M_t^{-a_{\tau}})\cdot M_s^{b_{\sigma}} M_{t}^{a_\tau} M_s^{-b_{\sigma}}M_t^{-a_{\sigma\tau}}\\
&=(1/d_{\sigma,\tau})\cdot M_t^{a_{\sigma}}M_t^{a_{\tau}}\cdot (M_t^{-a_{\tau}}\cdot M_s^{b_{\sigma}} \cdot M_{t}^{a_\tau}\cdot M_s^{-b_{\sigma}})\cdot M_t^{-a_{\sigma\tau}}\\
&=(1/d_{\sigma,\tau}) \cdot e^{\mathcal{L}}(-a_{\tau}t,b_{\sigma}s)\cdot M_t^{a_{\sigma}} M_t^{a_{\tau}} M_t^{-a_{\sigma\tau}}
\end{align*} Now we always have $M_t^{a_\sigma} M_t^{a_{\tau}}M_t^{-a_{\sigma\tau}}=\mathrm{Id}_{n\times n}$. This is clear if $a_\sigma +a_\tau<n$ and holds by our assumption $M_t^n=\mathrm{Id}_{n\times n}$ when $a_\sigma+a_\tau\geq n$. Altogether, we find that \begin{equation} \label{eq:leftside}\partial\circ {\Psi_{\mathcal{L}}}_*(\xi)(\sigma,\tau)=(1/d_{\sigma,\tau})\cdot e^{\mathcal{L}}(-a_{\tau}t,b_{\sigma}s)=(1/d_{\sigma,\tau})\cdot \left(\zeta_n^c\right)^{-a_{\tau}b_\sigma}\end{equation} for $\partial\circ{\Psi_{\mathcal{L}}}_*((\phi_n)_*^{-1}(0,[\beta]))(\sigma,\tau)=(1/d_{\sigma,\tau})\in (k^s)^\times$.

Keeping the notation above and going the other way around the diagram starts by \begin{align*} ev\circ (\phi_n^{-1})^\vee_*\left(e^{\mathcal{L}}(\bullet,-)_*(\xi)\right)(\sigma) & = ev\circ (\phi_n^{-1})^\vee\left(e^{\mathcal{L}}(a_\sigma t+b_\sigma s, -)\right)\\
 & = \left(e^{\mathcal{L}}(a_\sigma t+b_\sigma s,t), \log_{\zeta_n^c}\left(e^{\mathcal{L}}(a_\sigma t+b_\sigma s,s)\right)\right)\\
 &= ((\zeta_n^c)^{-b_\sigma},a_\sigma)\end{align*} so that $ev\circ (\phi_n^{-1})^\vee_*\left(e^{\mathcal{L}}(\bullet,-)_*(\xi)\right)=([\beta]^{-c},[\alpha])$. The cup-product in this case is the composition \begin{equation}\label{eq:cup}\cup:\mathrm{H}^1(G_k,\mu_n(k^s))\times \mathrm{H}^1(G_k,\mathbb{Z}/n\mathbb{Z})\xrightarrow{\cup_\otimes} \mathrm{H}^2(G_k,\mu_n(k^s)\otimes_{\mathbb{Z}} \mathbb{Z}/n\mathbb{Z})\rightarrow \mathrm{H}^2(G_k,\mu_n(k^s)).\end{equation} Here the first map is given, for any $\sigma,\tau\in G_k$, by \[([\beta]^{-c}\cup_\otimes [\alpha])(\sigma,\tau)=\zeta_n^{-cb_\sigma }\otimes \sigma{a_\tau}=\zeta_n^{-cb_\sigma }\otimes a_\tau.\] The second map in (\ref{eq:cup}) is multiplication, so \[ ([\beta]^{-c}\cup [\alpha])(\sigma,\tau)=(\zeta_n^c)^{-b_{\sigma}a_\tau}.\] The graded-commutativity of the cup-product implies that this is the same as $[\alpha]\cup[\beta]^c$. Finally, we observe from (\ref{eq:leftside}) and (\ref{eq:deltahereitis}) that \[\partial\circ\Psi_{\mathcal{L}_*}(\xi)=[\alpha]\cup [\beta]^c+\delta_s([\beta])=c([\alpha]\cup [\beta])+\delta_s([\beta])\] as claimed in the theorem statement.
\end{proof}

Recall that there is a bilinear \textit{determinant pairing} \begin{equation} \mathrm{det}:\left(\mathbb{Z}/n\mathbb{Z}\times \mu_n\right) \times \left(\mathbb{Z}/n\mathbb{Z}\times \mu_n\right) \rightarrow \mu_n \quad \quad \det((a,\zeta),(b,\xi))=\xi^a\zeta^{-b}.\end{equation} We say that $\phi_n:E[n]\cong \mathbb{Z}/n\mathbb{Z}\times \mu_n$ is a \textit{symplectic full level-$n$ structure} if the determinant pairing agrees with that of (\ref{eq:comm}) above under $\phi_n$, i.e.\ \begin{equation} e^{\mathcal{L}}(x,y)=\mathrm{det}(\phi_n(x),\phi_n(y))\end{equation} for all $x,y\in E[n](k^s)$.

Since $n$ is indivisible by the characteristic of $k$, the Weil-pairing (\ref{eq:comm}) induces an isomorphism of $G_k$-modules $\bigwedge^2 E[n](k^s)\cong \mu_n(k^s)$. Similarly, the determinant induces an isomorphism $\bigwedge^2(\mathbb{Z}/n\mathbb{Z}\times \mu_n(k^s))\cong \mu_n(k^s)$. By definition, the composition \begin{equation}\label{eq:sympcomp}\Phi_n: \mu_n(k^s)\cong \bigwedge^2 E[n](k^s)\xrightarrow{\bigwedge^2\phi_n(k^s)} \bigwedge^2 (\mathbb{Z}/n\mathbb{Z}\times \mu_n(k^s)) \cong \mu_n(k^s)\end{equation} is the identity if and only if $\phi_n$ is symplectic. By a direct check, we have $\Phi_n(\zeta^c_n)=\zeta_n$. Hence $c\equiv 1 \pmod{n}$ if and only if $\phi_n$ is symplectic. Thus we've shown:

\begin{corollary}[cf.\ {\cite[Propositions 2.13 and 2.14]{https://doi.org/10.48550/arxiv.2106.04291}}]
  \label{corollary:imageOfSplitting}
Let $\phi_n:E[n]\cong \mathbb{Z}/n\mathbb{Z}\times \mu_n$ be a symplectic full level-$n$ structure. Let $\xi\in \mathrm{H}^1(G_k,E[n](k^s))$ be any element. Then we have \[\Delta_t([\xi])=[\chi,b)_n+\delta_s([b])\] where $(\phi_n)_*([\xi])=(\chi,[b])\in \mathrm{H}^1(G_k,\mathbb{Z}/n\mathbb{Z})\times \mathrm{H}^1(G_k,\mu_n(k^s))$.$\hfill\square$
\end{corollary}

\begin{remark}
Given an arbitrary full level-$n$ structure $\phi_n:E[n]\cong \mathbb{Z}/n\mathbb{Z}\times \mu_n$, there is always an automorphism of the group scheme $\mathbb{Z}/n\mathbb{Z}\times \mu_n$ which will yield a symplectic full level-$n$ structure on composition with $\phi_n$. Specifically, if the map $\Phi_n$ from (\ref{eq:sympcomp}) above is exponentiation by some $1/c\in (\mathbb{Z}/n\mathbb{Z})^\times$ then the composition \[E[n]\xrightarrow{\phi}\mathbb{Z}/n\mathbb{Z}\times \mu_n\xrightarrow{c\times\mathrm{Id}} \mathbb{Z}/n\mathbb{Z}\times \mu_n\] is a symplectic full level-$n$ structure.
\end{remark}

\begin{corollary}\label{cor:cyclic}
Keep the notation of Theorem \ref{thm:antieauauel}. If $n$ is odd, then $\delta_s([b])=0$ and thus $\Delta_t([\xi])=c[\chi,b)_n$. On the other hand, if $n$ is even, then $\delta_s([b])=[\sigma,b)_2$ for some $\sigma\in \mathrm{H}^1(G_k,\mathbb{Z}/2\mathbb{Z})$. In this case we have \[\Delta_t([\xi])=c[\chi,b)_n+[\sigma,b)_2=[c\cdot\chi+(n/2)\circ \sigma,b)_n\] where $(n/2):\mathbb{Z}/2\mathbb{Z}\rightarrow \mathbb{Z}/n\mathbb{Z}$ is the inclusion-by-$(n/2)$ map. In particular, the obstruction $\Delta_t([\xi])$ is cyclic.
\end{corollary}

\begin{proof}
If $n$ is odd, then $\mathrm{H}^2(G_k,\mu_n(k^s))$ is $n$-torsion while $\delta_s([b])$ is $2$-torsion by Lemma \ref{lem: delta}. If $n$ is even, on the other hand, then let $L=k(\sqrt{b})$ be the field extension adjoining a square root of $b$ to $k$ and write $G_L$ for the absolute Galois group of $L$. Note that this field extension makes sense as, by our assumptions, when $n$ is even we are assuming that the characteristic of $k$ is not $2$.

We have the following commutative diagram from Lemma \ref{lem: delta}.
\[\begin{tikzcd}
\mathrm{H}^1(G_L,\mu_n(L^s))\arrow["\delta_s"]{r} &\mathrm{H}^2(G_L,\mu_n(L^s))\\
\mathrm{H}^1(G_k,\mu_n(k^s))\arrow["\delta_s"]{r}\arrow["res_k^L"]{u} &\mathrm{H}^2(G_k,\mu_n(k^s))\arrow["res_k^L"]{u}
\end{tikzcd}\] Now \[\mathrm{res}_k^L(\delta_s([b]))=\delta_s(\mathrm{res}_k^L([b]))=\delta_s([b])=2\delta_s([\sqrt{b}])=0.\]
Having $L/k$ as a splitting field shows that $\delta_s([b])\in \mathrm{H}^2(G_k,\mu_n(k^s))$ represents a quaternion algebra inside the Brauer group $\mathrm{Br}(k)$, see \cite[Proposition 2.5.3]{MR3727161}. Since the characteristic is not 2, there is an isomorphism of group schemes $\mu_2\cong \mathbb{Z}/2\mathbb{Z}$ in this case and so the Galois symbol is symmetric. We may therefore write $\delta_s([b])=[\sigma,b)_2$ for some $\sigma\in \mathrm{H}^1(G_k,\mathbb{Z}/2\mathbb{Z})$. The rest follows from Lemma \ref{lem:algebra}.
\end{proof}

Finer information than Theorem \ref{thm:antieauauel} is possible if one knows beforehand that $E[2n]$ also admits a full level-$2n$ structure $\phi_{2n}:E[2n]\cong \mathbb{Z}/2n\mathbb{Z}\times \mu_{2n}$.

\begin{theorem}\label{thm:2nstructure}
Suppose that $E[2n]$ admits a full level-$2n$ structure $\phi_{2n}:E[2n]\cong \mathbb{Z}/2n\mathbb{Z}\times \mu_{2n}$ which induces the full level-$n$ structure $\phi_n:E[n]\cong \mathbb{Z}/n\mathbb{Z}\times \mu_n$, i.e.\ assume that the diagram below is commutative \[\begin{tikzcd}E[2n]\arrow["\phi_{2n}"]{r}\arrow["q_E"]{d} & \mathbb{Z}/2n\mathbb{Z}\times \mu_{2n}\arrow["q\times p_2"]{d} \\ E[n]\arrow["\phi_{n}"]{r} & \mathbb{Z}/n\mathbb{Z} \times \mu_{n}  \end{tikzcd}\] where $q:\mathbb{Z}/2n\mathbb{Z}\rightarrow \mathbb{Z}/n\mathbb{Z}$ is the quotient map, $p_2:\mu_{2n}\rightarrow \mu_n$ is the squaring map, and $q_E:E[2n]\rightarrow E[n]$ is the multiplication-by-$2$ quotient. Fix also a primitive $2n$th root of unity $\zeta_{2n}$ with $\zeta_{2n}^2=\zeta_n$.

Then $\delta_s([b])=0$ for all $[b]\in k^\times/k^{\times n}\cong \mathrm{H}^1(G_k,\mu_{n}(k^s))$.
\end{theorem}

\begin{proof}
The case when $n$ is odd is covered by Corollary \ref{cor:cyclic} so we may assume that $n\geq 2$ is even throughout the proof. Let $t',s'\in E[2n](k^s)$ be such that $\phi_{2n}(t')=(1,1)$ and $\phi_{2n}(s')=(0,\zeta_{2n})$. Let $d\in \mathbb{Z}$ be such that $\zeta_{2n}^{d}=e^{2n}(t',s')$. There is then a symplectic full level-$2n$ structure $\phi_{2n}'$ fitting into the commutative triangle \[\begin{tikzcd} & \mathbb{Z}/2n\mathbb{Z}\times \mu_{2n}\arrow["{(1/d)}\times \mathrm{Id}"]{d}\\
E[2n]\arrow{r}\arrow["\phi'_{2n}"]{ur} & \mathbb{Z}/2n\mathbb{Z}\times \mu_{2n}. \end{tikzcd}\] We're going to use the following diagram.
\begin{equation}\label{eq:bigsquare}\begin{tikzcd}
\mathrm{H}^1(G_k,\mathbb{Z}/2n\mathbb{Z})\times \mathrm{H}^1(G_k,\mu_{2n}(k^s))\arrow["{(\frac{1}{d}\circ q)}_* \times (p_2)_*"]{d}\arrow["{({\phi'_{2n}}^{-1}})_*"]{r} & \mathrm{H}^1(G_k,E[2n](k^s))\arrow["{q_E}_*"]{d}\arrow["\Delta_{2n}"]{r} & \mathrm{H}^2(G_k,(k^s)^\times)\arrow["\cdot 2"]{d}\\
\mathrm{H}^1(G_k,\mathbb{Z}/n\mathbb{Z})\times \mathrm{H}^1(G_k,\mu_{n}(k^s))\arrow["{(\phi_n^{-1})}_*"]{r} & \mathrm{H}^1(G_k,E[n](k^s))\arrow["\Delta_t"]{r} & \mathrm{H}^2(G_k,(k^s)^\times)
\end{tikzcd}\end{equation}
The left square (\ref{eq:bigsquare}) is commutative by assumption and the right square of (\ref{eq:bigsquare}) is commutative due to the diagram \cite[$({*}{*}{*})_n$ on p.\ 310]{MR0204427} and the fact that $\mathcal{L}^{\otimes 2}=\mathcal{O}(p)^{\otimes 2n}$.

Now fix any element $b\in k^\times$ and let $x\in \mathrm{H}^1(G_k,E[n](k^s))$ be such that $(\phi_n)_*(x)=(0,[b])$. Note that the following diagram is commutative \[\begin{tikzcd} k^\times/k^{\times 2n}\arrow[equals]{d}\arrow["{[b]\mapsto [b]}"]{r} & k^\times/ k^{\times n}\arrow[equals]{d} \\ \mathrm{H}^1(G_k, \mu_{2n}(k^s))\arrow["{(p_2)}_*"]{r} & \mathrm{H}^1(G_k,\mu_n(k^s)) \end{tikzcd}\] where the vertical arrows are the K\"ummer isomorphisms. It follows that $(p_2)_*$ is surjective and, hence, there exists an element $y\in \mathrm{H}^1(G_k,E[2n](k^s))$ such that $(\phi'_{2n})_*(y)=(0,[b])$, i.e.\ so that ${q_E}_*(y)=x$. Then \[2\cdot \Delta_{2n}(y)=\Delta_t(x)=d[0,b)_n +\delta_s([b])=\delta_s([b])\] with the second to last equality coming from Theorem \ref{thm:antieauauel}. (The factor $d$ multiplying $[0,b)_n$ is irrelevant since $[0,b)_n=0$ but, it comes from the fact that \[e^{\mathcal{L}}(t,s)=e^{\mathcal{L}}(2t',2s')=e^{2n}(2t',s')=e^{2n}(t',s')^2=\zeta_{2n}^{2d}=\zeta_n^d\] by \cite[(4) on p.\ 228]{MR0282985}.)

Finally, by Remark \ref{rmk:quadratic} and \cite[Lemma 2.12]{https://doi.org/10.48550/arxiv.2106.04291} one finds \[2\cdot \Delta_{2n}(y)=e^{2n}_*(y\cup y)= 2[0,b)_{2n}=0.\] Hence $\delta_s([b])=0$ for all $b\in k^\times$ as claimed.
\end{proof}

\subsection{Splitting cyclic algebras}\label{ss: cyclicsplit}
In this subsection, we utilize the above results to give some interesting examples of genus one curves either splitting or embedding in Severi--Brauer varieties.

\begin{theorem}\label{thm:fulllevelinf}
Let $k$ be an arbitrary field. Let $F$ be a perfect field containing an algebraic closure $\overline{k}$ of $k$. Let $E/k$ be any elliptic curve and let $A/F$ be any central simple $F$-algebra.

Then there exists a principal homogeneous space $C$ for $E_F$ splitting $A$ if and only if $A$ is Brauer equivalent to a cyclic central simple $F$-algebra.
\end{theorem}

To cover the case when the characteristic of $F$ is $p>0$, we need the following lemma:

\begin{lemma}\label{lem:replacep}
Suppose that $F$ is a perfect field of characteristic $p>0$. Let $E$ be an elliptic curve over $F$. Let $A$ be a central simple $F$-algebra with associated Severi--Brauer variety $X=\mathrm{SB}(A)$. 

Then there exists a principal homogeneous space $C$ for $E$ which splits $A$ if and only if either of the following conditions hold:
\begin{enumerate}
\item $A$ is split, i.e.\ $A\cong M_n(F)$ and $X\cong \mathbb{P}^{n-1}$ for some $n\geq 1$;
\item or there exists a principal homogeneous space $C'$ for $E$, an integer $m$ coprime to $p$, and a degree-$m$ Severi--Brauer diagram $(C',\mu',f':C'\rightarrow Z)$ under $E$ with $Z$ Brauer equivalent to $X$.
\end{enumerate}
\end{lemma}

\begin{proof}
One direction is obvious. Also, if $A$ is split then $E$ clearly splits $A$. Suppose then that $C$ splits $A$ and that $A$ is a nonsplit central simple $F$-algebra. By the proof that \ref{prop: ds} implies \ref{prop: dsbd} in Proposition \ref{prop: splitandembedd} (which did not depend on the hypothesis $\mathrm{dim}(X)\geq 3$) there exists an integer $n\geq 2$ and a degree-$n$ Severi--Brauer diagram $\xi=(C,\mu,f:C\rightarrow Y)$ for a Severi--Brauer variety $Y$ that's Brauer equivalent to $X$. Since $F$ is perfect, there are no nontrivial central simple $F$-algebras of $p$-primary index \cite[Theorem 1.3.7]{MR4304038}. We can therefore write $n=dp^r$ where $d,r\in \mathbb{Z}$ are such that $\mathrm{gcd}(d,p)=1$ and $d$ is a multiple of the index of $A$. Let $[\xi]\in \mathrm{H}^1(G_F,E[n](F^s))$ be the equivalence class of $\xi$ under the bijection of Proposition \ref{prop:severiBrauerDiagramsBijections}.

Let $e\geq 1$ be an integer such that $ep^r \equiv 1 \pmod{\mathrm{exp}(A)}$. We will use the commutative diagram below.
\begin{equation}\label{eq:mumf}\begin{tikzcd}
\mathrm{H}^1(G_F,E[n](F^s))\arrow["\Delta_n"]{r}\arrow["\epsilon"]{d} & \mathrm{H}^2(G_F,(F^s)^\times)\arrow["\cdot e"]{d}\\
\mathrm{H}^1(G_F,E[ne](F^s))\arrow["\Delta_{ne}"]{r}\arrow["\eta"]{d} & \mathrm{H}^2(G_F,(F^s)^\times)\arrow["\cdot p^r"]{d}\\
\mathrm{H}^1(G_F,E[de](F^s))\arrow["\Delta_{de}"]{r} & \mathrm{H}^2(G_F,(F^s)^\times)
\end{tikzcd}\end{equation}
The top square of (\ref{eq:mumf}) results from the diagram \cite[$(**)_n$ on p.\ 309]{MR0204427} and the bottom square of (\ref{eq:mumf}) from the diagram \cite[$({*}{*}{*})_n$ on p.\ 310]{MR0204427}. Note that in \cite{MR0204427} Mumford assumes that the characteristic of the base field is coprime to the degree of the torsion subgroup of $E$ only in order to define the associated Theta group; by the characteristic free work of \cite[\S23]{MR0282985}, this assumption can safely be removed.

We get from (\ref{eq:mumf}) the equalities \[[A]=e p^r\cdot[A]=ep^r\cdot\Delta_n([\xi])=\Delta_{de}(\eta(\epsilon([\xi])))\] inside of $\mathrm{H}^2(G_F,(F^s)^\times)$. Hence by Proposition \ref{prop:obsforget} there exists a curve $C'$ with Jacobian $\mathbf{Pic}_{C'}^0=E$ and an equivalence class of a degree-$de$ Severi--Brauer diagram $\eta(\epsilon([\xi]))=[(C',\mu',f':C'\rightarrow Z)]$ for $E$ as claimed. 
\end{proof}

\begin{proof}[Proof of Theorem \ref{thm:fulllevelinf}]
Suppose that $C$ is a principal homogeneous space for $E_F$ splitting the central simple $F$-algebra $A$. Then by Proposition \ref{prop: splitandembedd}, there exists a degree-$n$ Severi--Brauer diagram $\xi=(C,\mu,f:C\rightarrow Y)$ under $E_F$ for a Severi--Brauer variety $Y$ Brauer equivalent to $X=\mathrm{SB}(A)$. By Lemma \ref{lem:replacep}, we may assume that $n$ is coprime to the characteristic of $F$ so long as we possibly also replace the Severi--Brauer diagram. Hence, we may assume that $\xi$ defines a class $[\xi]\in \mathrm{H}^1(G_F,E_F[n](F^s))$.

Since $E$ is defined over $k$, and since $F$ contains an algebraic closure $\overline{k}$ of $k$, we have that $E[n](\overline{k})=E[n](F)$. In particular, there is a full level-$n$ structure $\phi_n:E[n]\cong \mathbb{Z}/n\mathbb{Z}\times \mu_n$. Corollary \ref{cor:cyclic} then implies that $A$ is Brauer equivalent to a degree-$n$ cyclic algebra.
\end{proof}

\begin{example}\label{exmp: wadsworth}
Using Theorem \ref{thm:fulllevelinf}, we can now give examples of elliptic curves $E$ and of algebras $A$ such that no principal homogeneous space for $E$ splits $A$. Let $k$ be any field of characteristic zero, and let $L$ be any field containing an algebraic closure $\overline{k}$ of $k$. Let $F=L((t_1))((t_2))((t_3))((t_4))$ and let $A/F$ be the tensor product degree-$n$ symbol algebras $(t_1,t_2)_n\otimes_F(t_3,t_4)_n$ for any primitive $n$th root of unity $\zeta_n\in F$. Then, as we show below, $A$ is not split by any cyclic extension of $F$ and, so, $A$ is not Brauer equivalent to any cyclic algebra over $F$. Therefore, for any elliptic curve $E/k$ it follows that there are no principal homogeneous spaces for $E_F/F$ which split $A/F$. Note also that one can choose $L/k$ so that $E_F/F$ has infinitely many nontrivial principal homogeneous spaces of period $n^2$.

It remains to prove that $A$ is not split by any cyclic extension of $F$. The following proof was explained to the authors by Adrian Wadsworth, for which we are very grateful. Any errors are now our own. First, note that it suffices to prove the analogous claim when $L$ is algebraically closed. The field $F$ can be equipped with the iterated composite valuation $v_F:F^\times\rightarrow \mathbb{Z}^{\oplus 4}$ defined so that \[v_F(t_1^{m_1}t_2^{m_2}t_3^{m_3} t_4^{m_4})=(m_1,m_2,m_3,m_4)\quad \mbox{for all } m_1,m_2,m_3,m_4\in \mathbb{Z}.\] Here the value group $\Gamma_F=\mathbb{Z}^{\oplus 4}$ is given the reverse (right-to-left) lexicographic ordering. The valuation $v_F$ is Henselian \cite[Example A.16]{MR3328410} and so it extends to a unique valuation $v_A$ on the division algebra $A$ by \cite[Corollary 1.7]{MR3328410}. As $L$ is assumed algebraically closed, the algebra $A$ is totally ramified over $F$ with value group $\Gamma_A=\left(\frac{1}{n}\mathbb{Z}\right)^{\oplus 4}\supset \Gamma_F$. One can also define a canonical nondegenerate alternating pairing \[\Phi_A:(\Gamma_A/\Gamma_F)\times (\Gamma_A/\Gamma_F)\rightarrow \mu_n(L),\] by taking commutators of lifts of the corresponding elements in $A$, which determines $A$ up to isomorphism.

Any finite field extension $K/F$ inherits a unique valuation $v_K$ extending $v_F$ as well and, if $K/F$ is Galois, then $\mathrm{Gal}(K/F)\cong \Gamma_K/\Gamma_F$ since $K$ is totally ramified over $F$. Now the algebra $B:=A\otimes_F K$ is split if and only if $((\Gamma_A\cap \Gamma_K)/\Gamma_F)^{\perp}$ is contained in $\Gamma_K/\Gamma_F$ by \cite[Corollary 7.82]{MR3328410}. Here the intersection $\Gamma_A\cap \Gamma_K$ is taken inside the divisible hull of $\Gamma_F$ and the orthogonal is taken with respect to $\Phi_A$.

If $K$ is cyclic, then $\Gamma_K/\Gamma_F$ is a finite cyclic group. It follows that $(\Gamma_A\cap \Gamma_K\subset \Gamma_K)/\Gamma_F$ must be cyclic too. Hence $((\Gamma_A\cap \Gamma_K)/\Gamma_F)^\perp$ has order greater or equal $n^3$ but exponent $n$. In other words, $((\Gamma_A\cap \Gamma_K)/\Gamma_F)^\perp$ is contained in no cyclic group and so $B$ is not split.
\end{example}

In Remark \ref{rmk:nonellnorm} we gave examples of smooth genus one curves embedded inside Severi--Brauer varieties which could not be realized as geometrically elliptic normal curve inside the same Severi--Brauer variety. Here we give an example of an elliptic curve which appears as the Jacobian of a smooth genus one curve embedded inside a Severi--Brauer variety which does not admit any principal homogeneous spaces embedding as a geometrically elliptic normal curve into this same Severi--Brauer variety.

\begin{theorem}\label{thm:embed8}
Let $k=\mathbb{Q}_2$ and $F=\mathbb{Q}_2((t))$. Let $A$ be a division algebra with degree $\mathrm{deg}(A)=4$ over $k$. Let $B$ be the division $F$-algebra underlying the tensor product $A_F\otimes_F (-1,t)_2$. Then there exists an elliptic curve $E$ over $F$ such that \[[B]\notin \mathrm{Im}\left(\Delta_4:\mathrm{H}^1(G_F,E[4](F^s)\right)\rightarrow \mathrm{H}^2(G_F,(F^s)^\times)\quad\] and \[\quad[B]\in \mathrm{Im}\left(\Delta_8:\mathrm{H}^1(G_F,E[8](F^s))\rightarrow \mathrm{H}^2(G_F,(F^s)^\times)\right)\] where we write $G_F=\mathrm{Gal}(F^s/F)$ is the absolute Galois group of $F$.

In other words, there does not exist a principal homogeneous space for $E$ of period dividing $4$ which admits an embedding in $X=\mathrm{SB}(B)$ as a curve of degree $4$ but, there does exist a principal homogeneous space for $E$ of period dividing $8$ which embeds in $X$ as a curve of degree $8$.
\end{theorem}

\begin{proof}
We first consider the (smooth and projective) modular curve $X(8)$ defined over $\mathbb{Q}$. The modular curve $X(8)$ is a compactification of the smooth affine curve $Y(8)$ which is the fine moduli space parametrizing pairs $(E,\phi_8)$ of elliptic curves $E$ with a full symplectic level-$8$ structure $\phi_8$. 

It's known that there is at least one $\mathbb{Q}$-rational point in the complement $X(8)\setminus Y(8)$, i.e.\ a rational cusp. Hence $X(8)_k$ contains at least one smooth rational $k$-point $p$. If we write $K=k(X(8)_k)$ for the function field of $X(8)_k$, then it follows that this smooth $k$-point $p$ defines a valuation $v$ on $K$ for which the completion $K_v$ is isomorphic with $F$. We take the elliptic curve $E$ to be the elliptic curve in the pair $(E,\phi_8)$ associated to this $F$-point.

The algebra $B$ has been considered by Brussel in \cite{MR4072790}. There it's shown that $B$ is a noncyclic $F$-algebra of degree $4$, \cite[Theorem 3.2]{MR4072790}, and that $B$ has a cyclic splitting field of degree $8$, \cite[Remark 3.3]{MR4072790}. Now, by Remark \ref{rmk:deltatandn}, we have $\Delta_4=\Delta_t$ for any point $t\in E[4](F)$ of exact order 4. By Theorem \ref{thm:antieauauel} and Corollary \ref{cor:cyclic}, we have $[B]\notin \mathrm{Im}(\Delta_4)$ since $B$ is not itself cyclic. 

On the other hand, if $t'\in E[8](F)$ is a point of exact order $8$, then $[B]\in \mathrm{Im}(\Delta_{t'})$ because any degree-$8$ Galois symbol can be realized in the expression of Corollary \ref{cor:cyclic} and $M_2(B)$ is cyclic. By Corollary \ref{cor: splittoo}, it then follows that $[B]\in \mathrm{Im}(\Delta_8)$ as claimed.
\end{proof}

\begin{remark}
The embedding provided by Theorem \ref{thm:embed8} has the smallest non-minimal degree possible among all embeddings of genus one curves into $X$. Indeed, if a genus one curve embeds in $X$, then it also embeds as a geometrically elliptic normal curve into some Severi--Brauer variety $Y$ that's Brauer equivalent to $X$. Hence the degree of $C$ is some multiple of the index of $X$.
\end{remark}

\begin{remark}\label{rmk: embedpadic}
The construction of Theorem \ref{thm:embed8} is minimal in the sense that no such examples exist over any $p$-adic number field, so the completion with respect to the transcendental variable $t$ is needed. In fact, for any elliptic curve $E$ over any $p$-adic number field $k$, if $E$ admits a principal homogeneous space splitting a Severi--Brauer variety $X$ of index $d\geq 3$ over $k$, then $E$ admits a principal homogeneous space which embeds as a geometrically elliptic normal curve into every Severi--Brauer variety Brauer equivalent to $X$. 

Indeed, by Roquette's theorem \cite[Theorem 1]{MR0201435} a Severi--Brauer variety $X$ over such a $p$-adic field $k$ is split by a genus one curve $C$ if and only if the index of $X$ divides the index of $C$. Suppose then that $E$ is an elliptic curve admitting a principal homogeneous space $C$ which splits a Severi--Brauer variety $X$ of index $d\geq 3$. It follows that the index of $C$ is divisible by $d$. By \cite[Theorem 7]{MR0237506}, the period and index of $C$ are equal so that there exists a principal homogeneous space $C'$ for $E$ of period and index equal to $d$. Again by Roquette's theorem we know that $C'$ splits $X$ and hence (cf.\ Proposition \ref{prop: splitandembedd}) there exists a Severi--Brauer variety $Y$ which is Brauer equivalent to $X$ and of dimension $m-1$ for some $m$ divisible by $d$ such that $C'$ embeds in $Y$ of degree $m$.

Let $p\in \mathrm{Pic}^m_{C'}(k)$ be the point corresponding to the above embedding. Since $C'$ has index $d$, we can also consider a point $q\in \mathrm{Pic}_{C'}^d(k)$ corresponding to a line bundle associated to a Weil divisor of degree $d$ on $C'$. Let $m=de$ for some $e\geq 1$. If $e=1$, then already $C'$ can be realized as a geometrically elliptic normal curve inside every Severi--Brauer variety Brauer equivalent to $X$ by taking the embeddings corresponding to the points $p+rq$ varying over all integers $r\geq 0$. If $e>1$, then consider the point $t=p-(e-1)q\in \mathrm{Pic}^d_{C'}(k)$. Varying over integers $r\geq 0$, the points $t+rq$ also allow one to realize $C'$ embedded as a geometrically elliptic normal curve inside any Severi--Brauer variety Brauer equivalent to $X$, proving the claim.
\end{remark}

\begin{remark}\label{rmk:qpnstructure}
In \cite{https://doi.org/10.48550/arxiv.2106.04291}, Antieau and Auel show that over a global field base one may split central simple algebras of certain small degrees if one is willing to enlarge the base by various subfields of cyclotomic field extensions of the base. A simpler argument shows that, very similarly, one may split cyclic algebras over fields containing a $p$-adic field if one is willing to possibly adjoin certain roots of unity to the base field.

To illustrate this, let $n>1$ be any integer, let $p$ be a prime not dividing $n$, and let $E$ be an elliptic curve defined over the finite field $\mathbb{F}_p$. Then there is an extension $\mathbb{F}_q$ of $\mathbb{F}_p$ such that $E[n](\mathbb{F}_q)=E[n](\overline{\mathbb{F}}_p)$. More precisely, if $n=\prod_i p_i^{r_i}$ is a prime factorization of $n$, then one can take $q=p^r$ for some $r$ such that \[ r\,\,\bigg\vert \prod_i (p_i^{r_i-1})^4(p_i^2-1)(p_i^2-p_i)=\#\mathrm{GL}_2(\mathbb{Z}/n\mathbb{Z}).\] If $k=\mathbb{Q}_p(\zeta_{p^r-1})$ for a primitive $(p^r-1)$th root of unity $\zeta_{p^r-1}$, then for any field $F/k$, every cyclic central simple $F$-algebra of degree $n$ contains a smooth genus one curve.

Indeed, the modular curve $X(n)$, which is defined over $\mathrm{Spec}(\mathbb{Z}[1/n])$, contains an $\mathbb{F}_q$-point corresponding to the elliptic curve $E$ above. By Hensel's lemma, this implies $X(n)$ also contains a $k$-point corresponding to an elliptic curve $E'$ over $k$ with full level-$n$ structure. Hence, by Corollary \ref{cor:cyclic}, any cyclic degree-$n$ central simple algebra can be split over any field extension $F$ of $k$.

The above procedure will essentially always give a very crude bound, and such a large extension of the base won't be necessary to get the same effect. For example, if $p$ is large relative to $n$ then the modular curve $X(n)$ will already have an $\mathbb{F}_p$-point corresponding to an elliptic curve $E$ will full level-$n$ structure due to the Hasse-Weil bound (one notes that the number of cusps on $X(n)_{\mathbb{F}_p}$ is bounded independently of the prime $p$, so for large enough $p$ there must be a noncuspidal $\mathbb{F}_p$-point). Hensel's lemma then implies that there is an elliptic curve $E'$ with full level-$n$ structure defined over $\mathbb{Q}_p$ so long as $p$ is sufficiently large. Hence any cyclic degree-$n$ central simple algebra over any field containing such a $p$-adic local field will automatically contain a smooth curve of genus one by Corollary \ref{cor:cyclic}.
\end{remark}

\section{Conjugation}\label{sec: conj}
Let $E$ be an elliptic curve defined over a base field $k$. If all of the $n$-torsion points of $E$ are defined over $k$, for an integer $n\geq 2$ indivisible by the characteristic of the base field, then the absolute Galois group $G_{k}$ acts trivially on $E[n]$ and there is an isomorphism $\mathrm{H}^1(G_{k},E[n](k^s))\cong (k^{\times}/k^{\times n})^{\oplus 2}$.
  That is to say, combined with Corollary \ref{corollary:imageOfSplitting}, we understand degree-$n$ Severi--Brauer diagrams over $E$ in this case.
  However, this action is almost never trivial in practice.
  To deal with the scenario where the absolute Galois group does not act trivially, we consider degree-$n$ Severi--Brauer diagrams for $E$ over the field $F=k(E[n])$ where $E$ acquires all of its $n$-torsion points.
  To relate these diagrams to Severi-Brauer diagrams of $E$ over $k$, we consider the restriction map $\operatorname{res}:\mathrm{H}^{1}(G_{k},E[n](k^s))\to \mathrm{H}^{1}(G_{F},E[n](F^s))^{\mathrm{Gal}(F/k)}$.

\subsection{The inflation-restriction sequence}
  The following proposition describes the image of the restriction map in terms of certain Galois $(\mathbb{Z}/n\mathbb{Z}\times \mathbb{Z}/n\mathbb{Z})$-algebras over $F$.

  This proposition is useful in two ways.
  First, if we are in a setting where the restriction map is an injection (see Remark \ref{remark:casesWhereRestrictionInjective}) then it allows us to characterizes degree-$n$ Severi--Brauer diagrams of $E$ over $k$ after base-changing to $F=k(E[n])$.
  Second, the proposition shows that if a field $k$ does not admit Galois extensions of a certain shape then there cannot exist non-split Severi--Brauer diagrams.  
  
  \begin{proposition}
    \label{proposition:describingH1InTermsOfGaloisExtensions}
    Fix an integer $n\geq 3$ which is not divisible by the characteristic of $k$. Set $F=k(E[n])$.
    The image of the restriction map $$res_k^F:\mathrm{H}^1(G_{k},E[n](k^{s}))\rightarrow \mathrm{H}^1(G_{F},E[n](F^{s}))^{\mathrm{Gal}(F/k)}$$ is in bijection with the set of isomorphism classes of Galois $(\mathbb{Z}/n\mathbb{Z}\times \mathbb{Z}/n\mathbb{Z})$-algebras over $F$ which are naturally over $k$ also Galois $(\mathbb{Z}/n\mathbb{Z}\times \mathbb{Z}/n\mathbb{Z})\rtimes \mathrm{Gal}(F/k)$-algebras where $\mathrm{Gal}(F/k)$ acts on $\mathbb{Z}/n\mathbb{Z}\times \mathbb{Z}/n\mathbb{Z}$ by an isomorphism $\mathbb{Z}/n\mathbb{Z}\times \mathbb{Z}/n\mathbb{Z}\cong E[n](F)$ of $\mathrm{Gal}(F/k)$-modules.
  \end{proposition}

  \begin{proof}
    Since $G_{F}$ acts trivially on $E[n]$, we can realize $\mathrm{H}^1(G_F,E[n](F^s))$ as the set of isomorphism classes of Galois $(\mathbb{Z}/n\mathbb{Z}\times \mathbb{Z}/n\mathbb{Z})$-algebras $L$ over $F$ via \cite[(28.15) Example]{MR1632779}. If $t,s\in E[n](F)$ is a $\mathbb{Z}/n\mathbb{Z}$-basis for $E[n](F)$ and $\zeta_n\in F$ is a primitive $n$th root of unity, then this can be realized by the decomposition \[\mathrm{H}^1(G_F,E[n](F^s))\cong \mathrm{H}^1(G_F,\mathbb{Z}/n\mathbb{Z})\oplus \mathrm{H}^1(G_F,\mathbb{Z}/n\mathbb{Z})\cong \mathrm{H}^1(G_F,\mu_n(F^s))\oplus \mathrm{H}^1(G_F,\mu_n(F^s)).\] The rightmost group corresponds to pairs $(a,b)\in F^\times/F^{\times n}$ via Hilbert's theorem 90. Associated to such a pair $(a,b)$ is a pair of field extensions $F_a=F(a^{1/n})$ and $F_b=F(b^{1/n})$ which are naturally summands of $\mathbb{Z}/n\mathbb{Z}$-algebras $L_a,L_b$ constructed as in the exposition of subsection \ref{ss:cyclicalg}. The tensor product $L_{a,b}=L_a\otimes_F L_b$ is then the $(\mathbb{Z}/n\mathbb{Z}\times \mathbb{Z}/n\mathbb{Z})$-algebra over $F$ associated to the pair $(a,b)$.
    
    The action of $\mathrm{Gal}(F/k)$ by conjugation preserves such an algebra $L_{a,b}/F$ if and only if $L_{a,b}/k$ is also a $\mathrm{Gal}(F/k)$-algebra (which happens if and only if both $F_a/k$ and $F_b/k$ are Galois extensions of fields) and there is an isomorphism $\mathbb{Z}/n\mathbb{Z}\times \mathbb{Z}/n\mathbb{Z}\cong E[n](F)$ of $\mathrm{Gal}(F/k)$-modules which makes $L_{a,b}/k$ a Galois $G$-algebra for a group extension $G$ of $\mathrm{Gal}(F/k)$ by $\mathbb{Z}/n\mathbb{Z}\times \mathbb{Z}/n\mathbb{Z}$. Such an algebra $L_{a,b}/F$ is in the image of the restriction map if and only if the transgression of $L/F$ vanishes, i.e.\ if and only if $G\cong (\mathbb{Z}/n\mathbb{Z}\times \mathbb{Z}/n\mathbb{Z})\rtimes \mathrm{Gal}(F/k)$.
 \end{proof}

  \begin{remark}
    \label{remark:casesWhereRestrictionInjective}
    If the restriction map in Proposition \ref{proposition:describingH1InTermsOfGaloisExtensions} is injective, then one obtains a reasonably concrete description of $\mathrm{H}^1(G_{k},E[n](k^s))$ as a subset of $F^\times/F^{\times n}$.
  We give two natural settings where this happens.
    \begin{enumerate}
      \item{
        If $[F:k]$ and $n$ are coprime then the restriction map is an isomorphism.
      }
      \item{
        If $k=\mathbb{Q}$ and $n\geq 13$ is prime, then the restriction map is injective by \cite[Theorem 1]{MR3629657}.
      }
    \end{enumerate}
  \end{remark}
  
To give explicit examples of splitting of Severi--Brauer varieties, we first recall a characterization of abelian tamely ramified extensions of complete discretely valued fields with finite residue field.

\begin{lemma}
    \label{lemma:iffExtensionAbelian}
    Suppose $K$ is a complete discrete valued field with finite residue field, $k$, of characteristic $p$.
    Suppose $L/K$ is an extension with ramification index $e$. Suppose further that $(e,p)=1$.
    
    Then $L/K$ is abelian Galois if and only if $|k|\equiv 1\pmod{e}$ where $|k|$ is the number of elements in $k$. Equivalently, $L/K$ is abelian Galois if and only if $K$ contains a primitive $e$th root of unity.
\end{lemma}

  \begin{proof}
  See \cite[Chapter 16, p.\ 251]{MR0562104}.
  \end{proof}

\begin{theorem}\label{thm:quadtwist}
Let $n>2$ be an odd integer. Fix a prime number $p$ such that $p\neq 2$ and such that $p$ does not divide $n$. Let $k/\mathbb{Q}_p$ be a $p$-adic number field containing a primitive $n$th root of unity. Assume also that there is an elliptic curve $E$ over $k$ with full level $n$-structure. 

Let $L/k$ be a quadratic extension and write $E_L$ for the quadratic twist of $E$ with respect to $L$. Then there does not exist any nontrivial principal homogeneous space $C$ for $E_L$ of period dividing $n$.
\end{theorem}

\begin{proof}
It suffices to assume that $n$ is prime, possibly replacing $n$ with a prime divisor of $n$. By Remark \ref{rmk:Kummer}, it also suffices to show that $\mathrm{H}^1(G_k,E_L[n](k^s))=0$. The degree $[L:k]$ is relatively prime to $n$ and so the restriction map induces an isomorphism \[\mathrm{H}^1(G_k,E_L[n](k^s))\rightarrow \mathrm{H}^1(G_L,E_L[n](L^s))^{\mathrm{Gal}(L/k)}.\] According to Proposition \ref{proposition:describingH1InTermsOfGaloisExtensions}, a nontrivial element of $\mathrm{H}^1(G_L,E_L[n](L^s))^{\mathrm{Gal}(L/k)}$ corresponds to a  Galois $(\mathbb{Z}/n\mathbb{Z}\times \mathbb{Z}/n\mathbb{Z})\rtimes \mathbb{Z}/2\mathbb{Z}$-algebra $F/k$ where $\mathbb{Z}/2\mathbb{Z}$ acts on $\mathbb{Z}/n\mathbb{Z}\times \mathbb{Z}/n\mathbb{Z}$ as $-1$.

Contained in $F$ is a Galois field extension $K/k$ which has $\mathrm{Gal}(K/k)\cong D_{n}$ the dihedral group of order $2n$. But this is impossible since every extension of order $2n$ is cyclic by Lemma \ref{lemma:iffExtensionAbelian}.
\end{proof}

Using Theorem \ref{thm:quadtwist}, we can produce examples in the spirit of those following \cite[Corollary 3]{https://doi.org/10.48550/arxiv.2105.09986}.

\begin{corollary}\label{cor: exmpc}
For any odd integer $n>2$, there exist infinitely many primes $p>2$ such that the following conditions all hold simultaneously:
\begin{enumerate}
\item\label{cor:1unity} the $p$-adic field $\mathbb{Q}_p$ contains a primitive $n$th root of unity;
\item\label{cor:1curve} over $\mathbb{Q}_p$ there exists an elliptic curve $E$ with full level-$n$ structure;
\item\label{cor:1split} if $A/\mathbb{Q}_p$ is a central simple algebra with $\mathrm{ind}(A)$ dividing $n$, then $\mathrm{SB}(A)$ is split by a principal homogeneous space for $E$;
\item\label{cor:1nosplit} if $A/\mathbb{Q}_p$ is a central simple algebra with $\mathrm{ind}(A)$ dividing $n$ then, for any quadratic extension $L/\mathbb{Q}_p$, there does not exist any principal homogeneous space for the quadratic twist $E_L$ which splits $A$.
\end{enumerate}
\end{corollary}

\begin{proof}
Dirichlet's theorem on arithmetic progressions implies that there are infinitely many primes such that (\ref{cor:1unity}) is true. By Remark \ref{rmk:qpnstructure}, we also know that (\ref{cor:1curve}) is true for all large primes $p$. Now since all central simple algebras over $\mathbb{Q}_p$ are cyclic, Corollary \ref{cor:cyclic} implies (\ref{cor:1split}) for all such curves $E$ as in (\ref{cor:1curve}).

Finally, if $L/\mathbb{Q}_p$ is a quadratic extension, $E_L$ the associated quadratic twist, and $A/\mathbb{Q}_p$ a central simple algebra of period dividing $n$, then there exists a principal homogeneous space for $E_L$ which splits $A$ only if there exists a nontrivial principal homogeneous space for $E_L$ of period divisible by a factor of $n$ by Roquette's theorem \cite[Theorem 1]{MR0201435}. But this is impossible by Theorem \ref{thm:quadtwist}.
\end{proof}

\begin{remark}
For a fixed $p$-adic field satisfying (\ref{cor:1unity}) of Corollary \ref{cor: exmpc}, there will exist infinitely many curves $E$ satisfying (\ref{cor:1curve}) due to the ampleness of local fields. All such curves will also satisfy (\ref{cor:1split}) and (\ref{cor:1nosplit}). Hence one can produce both infinitely many $p$-adic fields and infinitely many curves for a given field as claimed in Theorem D of the introduction.
\end{remark}
  
\bibliographystyle{amsalpha}
\bibliography{bib}
\end{document}